\documentclass[12pt, leqno]{article}
\usepackage{amsmath}
\usepackage{amsthm}
\usepackage{amssymb}
\usepackage{latexsym}
\usepackage{graphicx}
\allowdisplaybreaks[1]
\usepackage{amsfonts}
\usepackage{ascmac}
\usepackage{color}
\usepackage{enumerate}

\theoremstyle{definition}
\theoremstyle{plain}

\allowdisplaybreaks[1]

\date{}
\newtheorem{Thm}{Theorem}[section]
\newtheorem{Prop}[Thm]{Proposition}

\newtheorem{Cor}[Thm]{Corollary}

\newtheorem{Rem}[Thm]{Remark}

\newcommand{\graph}{\mbox{\rm graph}}

\newcommand{\dis}{\displaystyle}

\newcommand{\norm}{\parallel}

\newcommand{\Z}{{\mathbb Z}}

\newcommand{\T}{{\mathbb T}}
\newcommand{\N}{{\mathbb N}}
\newcommand{\R}{{\mathbb R}}

\newcommand{\supp}{{\rm supp}\,}
\newcommand{\ep}{\varepsilon}

\def\text#1{\mbox{#1 }}

\title{\bf Weak KAM theory for discount Hamilton-Jacobi equations and its application}
\author{Hiroyoshi Mitake 
\footnote{Institute of Engineering, Division of Electrical, Systems and Mathematical Engineering, 
Hiroshima University, 1-4-1 Kagamiyama, Higashi-Hiroshima, 739-8527, Japan (hiroyoshi-mitake@hiroshima-u.ac.jp) Partially supported by JSPS grants: KAKENHI \#15K17574, \#26287024, \#16H03948.}\,\,\,
and Kohei Soga
\footnote{Department of Mathematics, Faculty of Science and Technology, Keio University, 3-14-1 Hiyoshi, Kohoku-ku, Yokohama, 223-8522, Japan (soga@math.keio.ac.jp). 
Partially supported by JSPS Grant-in-aid
for Young Scientists (B) \#15K21369.}}

\begin{document}
\maketitle
\begin{abstract}
\noindent Weak KAM theory for discount Hamilton-Jacobi equations and corresponding discount Lagrangian/Hamiltonian dynamics is developed. Then it is applied to error estimates for viscosity solutions in the vanishing discount process. The main feature is to introduce and investigate the family of $\alpha$-limit points of minimizing curves, with some details in terms of minimizing measures. In error estimates,  the family of $\alpha$-limit points is effectively exploited with properties of the corresponding dynamical systems.        

 \medskip

\noindent{\bf AMS subject classifications:} 
{\color{black}
35B40, %Asymptotic behavior of solutions, 
37J50, %Action-minimizing orbits and measures
49L25 %Viscosity solutions  
}
\end{abstract}
\setcounter{section}{0}
\setcounter{equation}{0}
\section{Introduction}
Weak KAM theory states the connection between viscosity solutions of the Hamilton-Jacobi equation with constant $c\in\R^n$ and $h(c )\in\R$,
\begin{eqnarray} \label{HJ}
H(x,c+v_x(x))=h(c)\mbox{\qquad in $\T^n$}
\end{eqnarray}
and the corresponding Hamiltonian (resp., Lagrangian) dynamics generated by $H$ (resp., the Legendre transform $L$ of $H$). The function $H$ is assumed to be a Tonelli Hamiltonian, i.e.,

(H1) $H(x,p):\T^n\times\R^n\to\R$, $C^2$, \quad\\
\indent (H2) $H_{pp}$ is positive definite,  \quad\\
\indent(H3) $\dis \lim_{|p|\to+\infty}\frac{H(x,p)}{|p|}=+\infty$ uniformly.

\noindent It is well-known that for each $c\in\R^n$ there exists the unique constant $h(c )$ for which (\ref{HJ}) admits a viscosity solution, where such a viscosity solution is not unique with respect to $c$ even up to constants. The function $c\mapsto h(c )$ is called the \textit{effective Hamiltonian}.  

There are several techniques to construct or approximate viscosity solutions of (\ref{HJ}) such as the vanishing viscosity method, a finite difference approximation and a discount approximation. In this paper, we consider a discount approximation of the form
\begin{eqnarray}\label{dHJ}   
\varepsilon v^\ep+H(x,c+v_x^\ep)=h(c )\mbox{\qquad in $\T^n$},
\end{eqnarray}
where $h(c )$ is same as the effective Hamiltonian in (\ref{HJ}) and $\ep>0$. The problem (\ref{dHJ}) is uniquely solvable for each $c$ because of the term ``$\ep v^\ep$'', which is sometimes called a \textit{discount factor} in the theory of optimal control. Recently it is proved that there exists a viscosity solution $v^\ast$ of (\ref{HJ}) such that $v^\ep\to v^\ast$ uniformly as $\ep\to0+$ based on weak KAM theory \cite{Davini} and on the nonlinear adjoint method \cite{MT} (\cite{MT} covers some degenerate elliptic problems). Note that this convergence result holds even in the case where (\ref{HJ}) has more than one solution beyond constant difference. 
Thus, this convergence itself is highly nontrivial, and 
is sometimes called the \textit{selection problem} in Hamilton-Jacobi equations. The selection criterion in the vanishing discount process is given in \cite{Davini}, \cite{MT}, but a rate of convergence is still an interesting open problem. 

We will give partial results on a rate of convergence, investigating weak KAM theory to discount Hamilton-Jacobi equations and the corresponding dynamical systems. The corresponding dynamical systems are Lagrangian/Hamiltonian dynamics with a friction term (\ref{dEL})/(\ref{dHS}), where they are still equivalent through the Legendre transform but there is not the Hamiltonian structure due to the friction term.  We call the dynamical systems the \textit{discount Lagrangian/Hamiltonian dynamics}. Variational characterization of $v^\ep$ by discount value functions is available as  (\ref{d-value-func}). Each minimizing curve of  (\ref{d-value-func}) yields an orbit of the discount Lagrangian/Hamiltonian dynamics. 
Then, discussion similar to standard weak KAM theory is available to discount Hamilton-Jacobi equations. 
Our investigation owes to the fact that discount Hamilton-Jacobi equations are of the class that admits the method of characteristics with deterministic dynamical systems. This is not the case with the vanishing viscosity method or a finite difference method, which makes the problem more difficult  in general.

One of main features of this paper is to focus on $\alpha$-limit points of minimizing curves to recover weak KAM theory, which does not seem to be done yet in the literature. Let us first recall original \textit{Mather's minimizing problem}: 
\begin{enumerate}
\item[(M)] {\it Consider }
\end{enumerate}
\begin{eqnarray}\label{p-mather}
\inf \int_{\T^n\times\R^n} L(x,\xi)-c\cdot\xi\,\,d\mu=-h(c), 
\end{eqnarray}
\begin{enumerate}
\item[] {\it where the infimum is taken over all probability measures on $\T^n\times\R^n$ which are invariant under the Euler-Lagrange flow of {\rm(\ref{EL})}.}
\end{enumerate}
A minimizing measure of (M) is called a \textit{Mather measure}, and the union of the supports of all Mather measures with each $c$ is called the {\it Mather set}.  There are many ways to construct a Mather measure, e.g., \cite{Mather1},  \cite{Fathi}, \cite{Mane}, \cite{EG}, \cite{MT}. 
General properties of the Mather set in regards to viscosity solutions are well-known. It is also known that each minimizing curve of a viscosity solution to (\ref{HJ}) induces a Mather measure, whose support is contained in the family of $\alpha$-limit points of the curve \cite{Fathi}, \cite{WE}.  
Furthermore, Poincar\'e's recurrence theorem states that, for each Mather measure $\mu$, $\mu$-a.e. points of $\supp(\mu)$ are recurrence.  Hence one can expect  inclusion between the Mather set and the family of  $\alpha$-limit points of all minimizing curves of all viscosity solutions to (\ref{HJ}).  However, since $\alpha$-limit points of minimizing curves are obtained viscosity-solution-wise, it is not trivial to find a general property of the family independently from a choice of a viscosity solution. 

Our first aim is to specify the family of  all $\alpha$-limit points of all minimizing curves for each $c$ to be one of the main objects in weak KAM theory such as the Mather set, Aubry set, Ma\~n\'e set, etc.  We show that the family of $\alpha$-limit points contains the Mather set and is contained in the Aubry set, summarizing its general properties.   

Our second aim is to generalize what we will observe in the non-discount problem to the discount problem, where complementary analysis of viscosity solutions and the corresponding dynamical systems with minimizing curves  is effectively done. Parallel argument with minimizing measures is also available. Indeed, we consider a minimizing problem associated with the discount problem: 
\begin{enumerate}
\item[(M1)$^\ep$] {\it Let $v^\ep$ be the {\rm(}unique{\rm) }viscosity solution of {\rm(\ref{dHJ})}. Consider 
$$\inf \int_{\T^n\times\R^n}L(x,\xi)-c\cdot\xi    -\ep v^\ep(x) \,d\mu=-h(c),$$
where the infimum is taken over all  probability measures on $\T^n\times\R^n$ which are invariant under the flow of the discount  Euler-Lagrange equation {\rm(\ref{dEL})}.  }\\
\end{enumerate}
This minimizing problem can be considered as a natural generalization of (M). The minimizing measures of (M1)$^\ep$ are, by definition, invariant under the flow of the discount Lagrangian dynamics, and the union of their support is contained in the family of $\alpha$-limit points of all minimizing curves of the viscosity solution to (\ref{dHJ}). We prove that the family of $\alpha$-limit points is obtained with properties similar to those in the non-discount problem. We also compare (M1)$^\ep$ with another type of a generalization of (M) to the discount problem: 
\begin{itemize}
\item[(M2)$^\ep$]   {\it Let $v^\ep$ be the viscosity solution of {\rm(\ref{dHJ})}. Consider for each $x_0\in\T^n$,
$$\inf \int_{\T^n\times\R^n}L(x,\xi)-c\cdot\xi\,d\mu=-h(c)+\ep v^\ep(x_0),$$
where the infimum is taken over all  probability measures on $\T^n\times\R^n$ which satisfy }
\end{itemize}
\begin{eqnarray}\label{d-holonomic}
\int_{\T^n\times\R^n}\varphi_x(x)\cdot\xi \,d\mu =\ep \varphi(x_0)-\ep\int_{\T^n\times\R^n}\varphi(x) d\mu\,\,\,\,\,\,\mbox{\it for all $\varphi\in C^1(\T^n,\R)$}. \\\nonumber
\end{eqnarray}
This problem is first introduced in \cite{G2005}, and then the selection problem in the vanishing discount process is solved with its minimizing measures \cite{Davini}, \cite{MT}.  Minimizing measures of  (M2)$^\ep$ are not invariant, and their support has no information on $\alpha$-limit points in general. Regardless of such difference between (M1)$^\ep$ and (M2)$^\ep$, both minimizing measures tend to some of minimizing measures of (M), i.e., Mather measures, 
as the discount  parameter $\ep$ goes to zero. 

Our third aim is to apply the analysis on the family of $\alpha$-limit points of minimizing curves  to an error estimate between $v^\ep$ and the unique limit $v^\ast$ of the vanishing discount process obtained in \cite{Davini}, \cite{MT}. We take the following strategy to obtain an estimate:  
\begin{enumerate}
\item[(i)] First, we estimate an error on the set of the family of $\alpha$-limit points. 
\item[(ii)] Second, we estimate the time for each point of $\T^n$ to 
fall into the family of $\alpha$-limit points along a minimizing curve of $v^\ep$ and $v^\ast$. 
\end{enumerate}
This strategy would work well, if properties of the corresponding dynamical systems are a priori known. In this paper, we show two successful examples. 

The first example is a simple one-dimensional problem whose corresponding Lagrangian dynamics and  Hamiltonian dynamics possesses hyperbolicity, where exponential asymptotics toward the family of $\alpha$-limit points is available. Since hyperbolicity is persistent for any small (non-Hamiltonian) perturbation, the discount Lagrangian/Hamiltonian dynamics can be still studied by means of hyperbolicity. 
In this discussion, we specify the families of  $\alpha$-limit points of minimizing curves of $v^\ep$ and $v^\ast$,  where we can also observe which Mather measures are obtained as the limit of minimizing measures of (M1)$^\ep$ through the vanishing discount process (selection of Mather measures). It is interesting to note that not every Mather measure is available. The difference between the selection criterion of the vanishing discount process and that of the vanishing viscosity process is also visible. See  Theorem \ref{d-Mather1} and  Remark \ref{same} for details.

The second example is that an exact viscosity solution admits a KAM torus. In this case, the family of $\alpha$-limit points is equal to  the whole set $\T^n$ and each minimizing curve is ergodic on $\T^n$ with a Diophantine rotation vector. We obtain an error estimate using a rate of ergodicity given by the Diophantine exponent. 

Although we demonstrate error estimates only in two special cases as a first step, our results imply that the error between $v^\ep$ and $v^\ast$ may depend highly on dynamics of the corresponding dynamical systems in general.

Finally we refer to recent development of analysis on the selection problem related to (\ref{HJ}). A generalization of \cite{Davini}, \cite{MT} to second-order fully nonlinear problems is given in \cite{I} introducing a dual method. A discrete version of \cite{Davini} is shown in \cite{Davini 2}.  Some non-convex cases are studied in \cite{GMT}.  A partial result to the selection problem in the vanishing viscosity method is obtained in \cite{Bessi}, \cite{A}, and that in a finite difference method is given in \cite{Soga4}. The lecture note  \cite{MT-Lecture} states lots about related topics based on the nonlinear adjoint method. Except for the case of the vanishing discount approximation in a convex setting, the selection problem on \eqref{HJ} is still rather open. 

The paper is organized as follows: In Section 2, we recall some of weak KAM theory for the non-discount problem and investigate the family of $\alpha$-limit points. In Section 3, we extend the argument in Section 2 to the discount problem, with details on (M1)$^\ep$ and (M2)$^\ep$. In Section 4, we demonstrate error estimates for $v^\ep$ and $v^\ast$ in the above two cases.     
 
\noindent {\bf Acknowledgement.} The authors would like to thank Professor Albert Fathi for valuable comments on this work. 
The authors are grateful to Professors Diogo A. Gomes and Yifeng Yu for valuable comments for the previous version of the manuscript.   
%%%%%%%%%%%%%%
\setcounter{section}{1}
\setcounter{equation}{0}
\section{Weak KAM theory for non-discount problem}
We overview weak KAM theory and investigate the family of $\alpha$-limit points of minimizing curves, which is denoted by $\mathcal{M}_\alpha(c)$. We show that $\mathcal{M}_\alpha(c)$ is a set between the (projected) Mather set $\mathcal{M}(c)$ and the (projected) Aubry set $\mathcal{A}(c)$.       
%%%%%%%%%%%%%
\subsection{Viscosity solution and minimizing curve}
In this section we recall several known facts on viscosity solutions of (\ref{HJ}) (see, e.g., \cite{Cannarsa} for more details). Let $L(x,\xi)$ be the Legendre transform of $H(x,p)$, which satisfies under (H1)--(H3), 
$$L(x,\xi)= \sup_{p\in\R^n}\{ \xi\cdot p -H(x,p)\}.$$
Let $v(x)$ be a viscosity solution of (\ref{HJ}). Note that $v$ is Lipschitz continuous. Then $v(x)$ satisfies for each $x\in\T^n$ and $T>0$,
\begin{eqnarray}\label{value-func}
v(x)=\inf_{\gamma\in AC,\gamma(0)=x}\left\{ \int^0_{-T}  ( L(\gamma(s),\gamma'(s))-c\cdot\gamma'(s) +h(c ) ) ds+v(\gamma(-T)) \right\},
\end{eqnarray} 
where $AC$ is the family of all absolutely continuous curves $[-T,0]\to\T^n$. By Tonelli's theory, one can find at least one minimizing curve $\gamma^\ast$ of the variational problem (\ref{value-func}), which is a $C^2$-solution of the Euler-Lagrange equation generated by $L^c:=L-c\cdot\xi$,
\begin{eqnarray}\label{EL}\quad\,\,\,
\frac{d}{ds} \{L^c_\xi(x(s),x'(s))\}=L^c_x(x(s),x'(s))\Leftrightarrow \frac{d}{ds} \{L_\xi(x(s),x'(s))\}=L_x(x(s),x'(s)).
\end{eqnarray}
Let $\phi^s_L$ denote the Euler-Lagrange flow of (\ref{EL}), i.e., $\phi^s_L(x(0),x'(0))=(x(s),x'(s))$. The viscosity solution $v$ is differentiable on the above minimizing curve $\gamma^\ast$ satisfying for $s\in[-T,0)$,
\begin{eqnarray}\label{nonderivative}
v_x(\gamma^\ast(s))=L^c_\xi(\gamma^\ast(s),\gamma^\ast{}'(s))=L_\xi(\gamma^\ast(s),\gamma^\ast{}'(s))-c.
\end{eqnarray}
In particular, if $v_x(x)$ exists, (\ref{nonderivative}) holds for $s=0$ and $\gamma^\ast$ is the unique minimizing curve for $v(x)$. The minimizing curve $\gamma^\ast$ can be extended to $(-\infty,0]$. By the variational property, we obtain for any $\tau>0$ and $\tilde{\tau}\ge\tau$, 
\begin{eqnarray*}
v(\gamma^\ast(-\tau))&=& \int^{-\tau}_{-\tilde{\tau}}  (L(\gamma^\ast(s),\gamma^\ast{}'(s))-c\cdot\gamma^\ast{}'(s) +h(c ))ds+v(\gamma^\ast(-\tilde{\tau}))\\
&=& \int^{0}_{-\tilde{\tau}+\tau}  (L(\gamma^\ast(s-\tau),\gamma^\ast{}'(s-\tau))-c\cdot\gamma^\ast{}'(s-\tau) +h(c ))ds+v(\gamma^\ast(-\tilde{\tau})),\\
v(x)&=&\int^{0}_{-\tau} ( L(\gamma^\ast(s),\gamma^\ast{}'(s))-c\cdot\gamma^\ast{}'(s) +h(c ))ds+v(\gamma^\ast(-\tau)),
\end{eqnarray*}
and hence we see that (\ref{nonderivative}) holds for all $s<0$. 

We call an extended minimizing curve defined on $(-\infty,0]$ a one-sided global minimizing curve or just a minimizing curve. Since $\gamma^\ast(s)$ solves (\ref{EL}) for $s\le 0$, we see that $(\gamma^\ast(s),p^\ast(s))$ with $p^\ast(s):=L_x(\gamma^\ast(s),\gamma^\ast{}'(s))$ is a $C^1$-solution of the Hamiltonian system,
\begin{eqnarray}\label{HS}
\left\{
\begin{array}{ll}
x'(s)=H_p(x(s),p(s)), \\ 
p'(s)=-H_x(x(s),p(s)).
\end{array}
\right. 
\end{eqnarray} 
%\begin{eqnarray}\label{HS}
%x'(s)=H_p(x(s),p(s)),\,\,\,p'(s)=-H_x(x(s),p(s)).
%\end{eqnarray} 
Let $\phi^s_H$ denote  the Hamiltonian flow of (\ref{HS}).  Due to equivalence between (\ref{EL}) and (\ref{HS}), and (\ref{nonderivative}), we have $p^\ast(s)=c+v_x(\gamma^\ast(s))$ for all $s< 0$, which means that for all $s<0$, 
$$(\gamma^\ast(s),p^\ast(s))\in \graph (c+v_x):=\{ (x,c+v_x(x))\,|\,x\in \T^n \mbox{ such that $v_x(x)$ exists}\}.$$  
Therefore we have 
%%%%%%%%%%
\begin{Thm}
The set $\graph(c+v_x)$ is backward invariant under the Hamiltonian flow $\phi^s_H$, i.e., $\phi^s_H(\graph(c+v_x))\subset \graph(c+v_x)\mbox{\quad for all  $s
\le0$}$. 
\end{Thm}
%%%%%%%
\subsection{$\alpha$-limit point}
Let $\gamma^\ast$ be a one-sided global minimizing curve. Consider $\alpha$-limit points of $\gamma^\ast$, where $x^\ast\in\T^n$ is called an $\alpha$-limit point if there exists a monotone sequence $\tau_j\to-\infty$ as $j\to\infty$ for which $\gamma^\ast(\tau_j)\to x^\ast$ as $j\to\infty$.  Fix $c$ and define the set for each viscosity solution $v$ of (\ref{HJ})  as
\begin{eqnarray*}
\,\,\,\,\,\mathcal{M}_\alpha(v;c ):=\overline{\{ x^\ast\in\T^n\,\,|\,\,\mbox{$x^\ast$: $\alpha$-limit point of $\gamma^\ast$, $\gamma^\ast$: minimizer for $v(x)$, $x\in\T^n$  }  \}},
\end{eqnarray*}
and take their union,
\begin{eqnarray*}
\mathcal{M}_\alpha(c ):=\overline{\bigcup_v \mathcal{M}_\alpha(v;c)},
\end{eqnarray*}
where 
$\overline{A}$ stands for the closure of $A\subset\R^m$ for $m\in\N$, and the union is taken over all the viscosity solutions of (\ref{HJ}) with fixed $c$. Here are properties of $\mathcal{M}_\alpha(v;c)$:
%%%%%%%%%%%%%%%
\begin{Thm}\label{invariance}
\begin{enumerate}[{\rm(i)}]
\item Each viscosity solution $v$ to \eqref{HJ} is differentiable on ${\mathcal{M}_\alpha}(v;c)$. 
\item Let $\gamma^\ast$ be a minimizing curve of $v$ and $(x^\ast,\xi^\ast)$ be an $\alpha$-limit point of $(\gamma^\ast(s),\gamma^\ast{}'(s))$. Then $\xi^\ast=H_p(x^\ast,c+v_x(x^\ast))$.
\item Let $x^\ast$ be a point of $\mathcal{M}_\alpha(v;c)$ and $x(s)$ be the solution of the Euler-Lagrange equation (\ref{EL}) with $x(0)=x^\ast$, $x'(0)=H_p(x^\ast,c+v_x(x^\ast))$. Then $x(s)\in\mathcal{M}_\alpha(v;c)$ and $v_x(x(s))=L_\xi(x(s),x'(s))-c$ for all $s\in\R$. 
\item Let  $v$ and $\tilde{v}$ be  two viscosity solutions of \eqref{HJ}. If $v=\tilde{v}$ on $\mathcal{M}_\alpha(c )$, then $v=\tilde{v}$ on $\T^n$.
\end{enumerate}
\end{Thm}
%%%%%%%%%%%%%%%
\begin{proof}
Let $x^\ast$ be an arbitrary point of $\mathcal{M}_\alpha(v;c)$ such that there exist a minimizing curve $\gamma^\ast:(-\infty,0]\to\T^n$ and $\tau_j\to-\infty$ ($j\to\infty$) satisfying $\gamma^\ast(\tau_j)\to x^\ast$ ($j\to\infty$). Since $\gamma^\ast{}'(\tau_j)=H_p(\gamma^\ast(\tau_j),c+v_x(\gamma^\ast(\tau_j)))$ is a bounded sequence, there exists a subsequence, still denoted by $\gamma^\ast{}'(\tau_j)$, which converges to some $\xi^\ast\in\R^n$. Let $x(s):\R\to\T^n$ be the solution of (\ref{EL}) with the initial condition $x(0)=x^\ast,\,x'(0)=\xi^\ast$. It follows from the continuous dependence on initial conditions that for any $\alpha>0$, 
$$(\gamma^\ast(\cdot+\tau_j), \gamma^\ast{}'(\cdot+\tau_j))\to (x(\cdot),x'(\cdot))\mbox{ uniformly on $[-\alpha,\alpha]$ ($j\to\infty$)}.$$   
Hence, we see that $x(s)\in\mathcal{M}_\alpha(v;c)$ for all $s\in\R$ by the definition of $\mathcal{M}_\alpha(v;c)$. Since $\gamma^\ast(\cdot+\alpha+\tau_j):[-2\alpha,0]\to\T^n$ is the unique minimizing curve for $v(\gamma^\ast(\alpha+\tau_j))$ for each fixed large $j$, i.e.,
\begin{eqnarray*}
v(\gamma^\ast(\alpha+\tau_j))&=&\int^0_{-2\alpha} (L(\gamma^\ast(s+\alpha+\tau_j),\gamma^\ast{}'(s+\alpha+\tau_j))-c\cdot\gamma^\ast{}'(s+\alpha+\tau_j) \\
&&+h(c ))ds+v(\gamma^\ast(-\alpha+\tau_j)),
\end{eqnarray*}
we obtain by letting $j\to\infty$,
\begin{eqnarray*}
v(x(\alpha))=\int^0_{-2\alpha}(L(x(s+\alpha),x'(s+\alpha))-c\cdot x'(s+\alpha) +h(c ))ds
+v(x(-\alpha)). 
\end{eqnarray*}  
Therefore, we see that $x(s+\alpha)$, $s\in[-\alpha,\alpha]$ is a minimizing curve for $v(x(\alpha))$ and thus  $v_x(x(\alpha+s))=L_\xi(x(\alpha+s),x'(\alpha+s))-c$ for all $s\in[-2\alpha,0)$ due to (\ref{nonderivative}). For $s=-\alpha$, we obtain $v_x(x^\ast)=L_\xi(x^\ast,\xi^\ast)-c$ and $\xi^\ast=x'(0)=H_p(x^\ast,c+v_x(x^\ast))$. 

If $x^\ast$ is an accumulating point of $\mathcal{M}_\alpha(v;c)$, we have a sequence $\{(x^\ast_i,\xi_i^\ast)\}$ of $\alpha$-limit points of minimizing curves $(\gamma^\ast_i(s),\gamma^\ast{}'_i(s))$ such that  $(x^\ast_i,\xi_i^\ast)\to(x^\ast,\xi^\ast)$ as $i\to\infty$. Let $x(s)$ be the solution of (\ref{EL}) with the initial condition $(x^\ast,\xi^\ast)$. Since the above argument holds for each $(x^\ast_i,\xi_i^\ast)$, the continuous dependence yields the same result for $x(s)$.   Noting that $\alpha>0$ is arbitrary, we complete the proof of (i) to (iii). 

We prove (iv). Let $x$ be an arbitrary point of $\T^n$. Let $\gamma^\ast$ (resp., $\tilde{\gamma}^\ast$) be a minimizing curve for $v(x)$ (resp., $\tilde{v}(x)$). There exists $\tau_j\to-\infty$ ($j\to\infty$) and $x^\ast\in\mathcal{M}_\alpha(c )$ such that $\gamma^\ast(\tau_j)\to x^\ast$ ($j\to\infty$) (resp., $\tilde{\tau}_j\to-\infty$ ($j\to\infty$) and $\tilde{x}^\ast\in\mathcal{M}_\alpha(c )$ such that $\tilde{\gamma}^\ast(\tilde{\tau}_j)\to \tilde{x}^\ast$ ($j\to\infty$)).  It follows from the variational representation formula (\ref{value-func}) that for each $j$,
\begin{eqnarray*}
\tilde{v}(x)-v(x)\le \tilde{v}(\gamma^\ast(\tau_j))-v(\gamma^\ast(\tau_j))\quad\mbox{(resp., $\tilde{v}(x)-v(x)\ge \tilde{v}(\tilde{\gamma}^\ast(\tilde{\tau}_j))-v(\tilde{\gamma}^\ast(\tilde{\tau}_j))$)}.
\end{eqnarray*}  
Since $v(x^\ast)=\tilde{v}(x^\ast)$ (resp., $v(\tilde{x}^\ast)=\tilde{v}(\tilde{x}^\ast)$), we conclude $v(x)\le\tilde{v}(x)$ (resp., $v(x)\ge\tilde{v}(x)$) by letting $j\to\infty$.  
\end{proof}
By Theorem \ref{invariance}, the following sets are well-defined: 
\begin{align*}
&\tilde{\mathcal{M}}_\alpha(v;c):=\{ (x,H_p(c+v_x(x)))\,|\,x\in\mathcal{M}_\alpha(v;c) \},
&\tilde{\mathcal{M}}_\alpha(c):=\bigcup_v\tilde{\mathcal{M}_\alpha}(v;c),\\
&\tilde{\mathcal{M}}_\alpha^\ast(v;c):=\{ (x,c+v_x(x))\,|\,x\in\mathcal{M}_\alpha(v;c) \},
&\tilde{\mathcal{M}}^\ast_\alpha(c):=\bigcup_v\tilde{\mathcal{M}^\ast_\alpha}(v;c),
\end{align*}
%\begin{eqnarray*}
%\tilde{\mathcal{M}}_\alpha(v;c)&:=&\{ (x,H_p(c+v_x(x)))\,|\,x\in\mathcal{M}_\alpha(v;c) \},\\
%\tilde{\mathcal{M}}_\alpha(c)&:=&\bigcup_v\tilde{\mathcal{M}_\alpha}(v;c),\\
%\tilde{\mathcal{M}}_\alpha^\ast(v;c)&:=&\{ (x,c+v_x(x))\,|\,x\in\mathcal{M}_\alpha(v;c) \},\\
%\tilde{\mathcal{M}}^\ast_\alpha(c)&:=&\bigcup_v\tilde{\mathcal{M}^\ast_\alpha}(v;c),
%\end{eqnarray*}
where the union  in the above is taken over all the viscosity solutions of (\ref{HJ}) with fixed $c$. 
%%%%%
\begin{Rem}  The curve $x(s)$ in Theorem \ref{invariance} is a global minimizing curve of $v$, which is not a homoclinic/heteroclinic orbit. The set $\mathcal{M}_\alpha(v;c)$ is equal to the closure of the family of all such global minimizing curves of $v$,  and
\begin{eqnarray*}
\tilde{\mathcal{M}}_\alpha(v;c)&=&\overline{\{ (x^\ast,\xi^\ast)\,|\, \mbox{$(x^\ast,\xi^\ast)$: $\alpha$-limit point of $(\gamma^\ast(s),\gamma^\ast{}'(s))$}}, \\
&&\qquad\qquad\qquad\qquad\qquad\qquad \overline{\mbox{$\gamma^\ast$: minimizing curve of $v$ } \}}.
\end{eqnarray*} 
\end{Rem}
%%%%%%%%%
\noindent Here are properties of $\tilde{\mathcal{M}}_\alpha(v;c)$ and $\tilde{\mathcal{M}}^\ast_\alpha(v;c)$:
%%%%%%%%%
\begin{Thm}\label{similar}
\begin{enumerate}[{\rm(i)}]
\item For each point $x\in\T^n$, there exists $\xi\in\R^n$ such that $\phi^s_L(x,\xi)$ falls into $\tilde{\mathcal{M}}_\alpha(v;c)$ as $s\to-\infty$, i.e., any accumulating point of $\{\phi^s_L(x,\xi)\}_{s\le0}$ belongs to $\tilde{\mathcal{M}}_\alpha(v;c)$. If $v_x(x)$ exists, $\xi=H_p(x,c+v_x(x))$.
\item  $\tilde{\mathcal{M}}_\alpha(v;c)$ is a $\phi^s_L$-invariant subset of $\T^n\times\R^n$, i.e., $\phi^s_L(\tilde{\mathcal{M}}_\alpha(v;c))=\tilde{\mathcal{M}}_\alpha(v;c)$  for all $s\in\R$. 
\item For each point $x\in\T^n$,  there exists $p\in\R^n$ such that $\phi^s_H(x,p)$ falls into $\tilde{\mathcal{M}}^\ast_\alpha(v;c)$ as $s\to-\infty$ along $\graph(c+v_x)$. If $v^\ep_x(x)$ exists, $p=c+v_x^\ep(x)$.
\item $\tilde{\mathcal{M}}^\ast_\alpha(v;c)$ is a $\phi^s_H$-invariant subset of $\graph(c+v_x)$. 
\end{enumerate}
\end{Thm}
%%%%%%%%%%%
\begin{proof}
(i) is clear by definition. (ii) is clear by (iii) of Theorem \ref{invariance}. Set $p:=L_\xi(x,\xi)$, $\xi:=\gamma^\ast{}'(0)$ with a minimizing curve $\gamma^\ast$ for $v(x)$. Then we have $\phi^s_H(x,p)=(\pi\circ \phi^s_L(x,\xi), L_\xi(\phi^s_L(x,\xi)))$, 
where $\pi:\T^n\times\R^n\to\T^n$ is the standard projection. 
Hence, (iii) follows from above (i), (\ref{nonderivative}) with equivalence between (\ref{EL}) and (\ref{HS}), and (ii) of Theorem \ref{invariance}. Let $(x^\ast,p^\ast)$ be an arbitrary point of $\tilde{\mathcal{M}}_\alpha^\ast(v;c)$. (iv) follows from  (iii) of Theorem \ref{invariance} with $x'(0)=H_p(x^\ast,p^\ast)$ and equivalence between (\ref{EL}) and (\ref{HS}).
\end{proof}
%%%%%%%%%%%
\subsection{Comparison with Mather set and Aubry set}
In this subsection, we compare $\mathcal{M}_\alpha(c )$, $\tilde{\mathcal{M}}_\alpha(c )$ with the (projected) Mather set  and the (projected) Aubry set defined in weak KAM theory.  

We recall the measure-theoretical aspect of weak KAM theory, which is first investigated in \cite{Mather1}. Let $\gamma^\ast:(-\infty,0]\to\T^n$ be a minimizing curve for $v(x)$. Note that $(\gamma^\ast(s),\gamma^\ast{}'(s))$ is contained in a compact set $K\subset\T^n\times\R^n$ for each $s\le0$, because $\{v^\ep\}_{\ep>0}$ is equi-Lipschitz continuous. Define the linear functional for each $T>0$ 
\begin{eqnarray}\label{holo110}
\Psi_T(\varphi):=\frac{1}{T}\int^0_{-T}\varphi(\gamma^\ast(s),\gamma^\ast{}'(s))ds,\,\,\,\,\varphi\in C_c(\T^n\times\R^n,\R),
\end{eqnarray}
where $C_c(\T^n\times\R^n,\R)$ denotes the family of compactly supported continuous functions. Then, by the Riesz representation theorem, there exists a probability measure $\mu_T$ on $\T^n\times\R^n$ such that 
\begin{eqnarray}\label{holo11}
\Psi_T(\varphi)=\int_{\T^n\times\R^n}\varphi(x,\xi)d\mu_T\,\,\,\mbox{ for all $\varphi\in C_c(\T^n\times\R^n,\R)$}.
\end{eqnarray}
Since the support of $\mu_T$, denoted by $\supp(\mu_T)$, is contained in $K$ for all $T>0$, we have $T_i\to\infty$  as $i\to\infty$ for which $\mu_{T_i}$ converges weakly to a probability measure  $\mu^\ast$. We see that $\mu^\ast$ is $\phi^s_L$-invariant. Furthermore, (\ref{holo11}) with $\varphi(x,\xi)=L(x,\xi)-c\cdot\xi+h(c )$ 
(more precisely, the right-hand-side is re-defined to be $0$ continuously outside $K$ so that it belongs to $C_c(\T^n\times \R^n, \R)$, which does not change anything in regards to a measure supported in $K$)
%more precisely, the right-hand-side is to be defined as $0$ outside $K$) 
yields 
\begin{eqnarray*}     
\int_{\T^n\times\R^n}L(x,\xi)-c\cdot\xi+h(c ) \,\,d\mu_T=\frac{1}{T}(v(\gamma^\ast(0))-v(\gamma^\ast(-T))). 
\end{eqnarray*}
Putting $T=T_i$ and sending $i\to\infty$, we obtain
\begin{eqnarray*}     
\int_{\T^n\times\R^n}L(x,\xi)-c\cdot\xi \,\,d\mu^\ast=-h(c ).
\end{eqnarray*}
We observe the minimizing property of $\mu^\ast$. For any $(x,\xi)\in\T^n\times\R^n$ and $t>\tau>0$, we have 
\begin{eqnarray}\label{22invariant-measure}
v(\pi\circ\phi^t_L(x,\xi))&\le& \int^t_{-\tau} L^c(\phi^s_L(x,\xi))+h(c )\, ds+v(\pi\circ\phi^{-\tau}_L(x,\xi))\\\nonumber
&=&\int^0_{-\tau-t} L^c(\phi^{s+t}_L(x,\xi))+h(c )\,ds+v(\pi\circ\phi^{-\tau}_L(x,\xi)). 
\end{eqnarray}
Then, integrating the inequality with any $\phi^s_L$-invariant probability measure $\mu$ defined on $\T^n\times\R^n$, we have 
$$0\le \int_{\T^n\times\R^n} L^c(x,\xi)+h(c )\,d\mu\Leftrightarrow \int_{\T^n\times\R^n}L(x,\xi)-c\cdot\xi \,\,d\mu\ge-h(c ). $$
Therefore we conclude that $\mu^\ast$ is a minimizing measure for the minimizing problem (M) in Introduction.   
 
Taking $\phi(x,\xi):=\psi_x(x)\cdot\xi$ (in $K$, otherwise re-defined to be $0$ continuously) with $\psi\in C^1(\T^n)$ in (\ref{holo11}), we have 
\begin{eqnarray*}
\int_{\T^n\times\R^n} \psi_x(x)\cdot\xi d\mu_T=\frac{1}{T}\int^0_{-T}\psi_x(\gamma^\ast(s))\cdot\gamma^\ast{}'(s)ds=\frac{1}{T}\{\psi(\gamma^\ast(0))-\psi(\gamma^\ast(-T))\}. 
\end{eqnarray*} 
With $T=T_j$ and $j\to\infty$, we see that $\mu=\mu^\ast$ satisfies  
\begin{eqnarray}\label{holonomic}
\int_{\T^n\times\R^n} \psi_x(x)\cdot\xi \,d\mu=0\mbox{ \,\,\,for all $\psi\in C^1(\T^n)$}.
\end{eqnarray}
A probability measure $\mu$ satisfying (\ref{holonomic}) is said to be \textit{holonomic}. It is proved that the minimizing problem (\ref{p-mather}) within all $\phi^s_L$-invariant measures is equivalent to the one within all holonomic measures \cite{Mane}. 

The (projected) Mather set $\mathcal{M}(c):=\pi\tilde{\mathcal{M}}(c)$, $\tilde{\mathcal{M}}(c)$ is defined as 
$$\tilde{\mathcal{M}}(c):=\overline{\bigcup_{\mu} \supp(\mu)},$$
where the union is taken over  all minimizing measures of (M).  As shown in \cite{Fathi-book}, \cite{EG},  for each point $(x,\xi)\in \tilde{\mathcal{M}}(c)$, any viscosity solution $v$ of (\ref{HJ}) is differentiable at $x$ possessing the common derivative 
\begin{eqnarray}\label{mather-deri}
v_x(x)=L_\xi(x,\xi)-c.
\end{eqnarray} 
In fact,  since (\ref{22invariant-measure})  holds for each $(x,\xi)\in\T^n\times\R^n$, it must be an equality on the support of each minimizing measure $\mu$, because otherwise our integration with $\mu$ yields $0<\int_{\T^n\times\R^n}L^c(x,\xi)d\mu+h(c)$. Continuity implies that  (\ref{22invariant-measure}) is an equality for each $(x,\xi)\in\tilde{\mathcal{M}}(c)$. Therefore $\pi\circ\phi_L^{\cdot+t}(x,\xi):[-\tau-t,0]\to\T^n$ yields a minimizing curve for $v(\phi_L^t(x,\xi))$ and $v$ is differentiable at $\pi\circ\phi_L^0(x,\xi)=(x,\xi)$ with (\ref{mather-deri}).     

The (projected) Aubry set  $\mathcal{A}(c ):=\pi\tilde{\mathcal{A}}(c )$, $\tilde{\mathcal{A}}(c )$ is defined through conjugate pairs of weak KAM solutions (see \cite{Fathi-book} for detail). By definition, for each point $(x,\xi)\in \tilde{\mathcal{A}}(c )$, any viscosity solution $v$ of (\ref{HJ}) is differentiable at $x$ possessing the common derivative  (\ref{mather-deri}). The following characterization of the projected Aubry set \cite{Fathi-book} is useful:
\begin{eqnarray*}
&&\mathcal{A}(c )=\{x\in\T^n\,|\,\mathfrak{h}^c(x,x)=0\},\\
&&\mathfrak{h}^c(x,y):=\liminf_{T\to\infty}\inf_{\gamma\in AC,\gamma(0)=x,\gamma(-T)=y}\int^0_{-T}L(\gamma(s),\gamma'(s))-c\cdot\gamma'(s)+h(c ) \,ds,
\end{eqnarray*}
where $\mathfrak{h}^c$ is the Peierls barrier. 
%%%%%%
%%%%%%%%%%%
\begin{Thm}\label{comparison} 
\begin{enumerate}[{\rm(i)}]
\item Let  $\mu$ be any minimizing measure of {\rm(M)}. Then, every point $(x,\xi)\in \supp(\mu)$ itself is an $\alpha$-limit point of $\phi^s_L(x,\xi)$. 
\item $\mathcal{M}(c )\subset\mathcal{M}_\alpha(c )\subset\T^n,\,\,\,\tilde{\mathcal{M}}(c )\subset\tilde{\mathcal{M}}_\alpha(c )\subset\T^n\times\R^n.$
\item $\mathcal{M}_\alpha(c )\subset\mathcal{A}(c )\subset\T^n,\,\,\,\tilde{\mathcal{M}}_\alpha(c )\subset\tilde{\mathcal{A}}(c )\subset\T^n\times\R^n.$
\end{enumerate}
\end{Thm}
%%%%%%%%%%%%%
\begin{proof}
It follows from Poincar\'e's recurrence theorem that $\mu$-a.e. points of $\supp(\mu)$ are recurrent, which yields (ii). A slightly more detailed argument is necessary to prove that {\it every} point is recurrent. We show a direct proof of (i) and (ii).   

Let $\mu$ be any minimizing measure of (M). As we already observed, for any $(x,\xi)\in\T^n\times \R^n$, we have (\ref{22invariant-measure}), which must be an equality on $\supp(\mu)$. Hence,  for each $(x,\xi)\in\supp(\mu)$, the curve $\pi\circ\phi^s_L(x,\xi):[-T,0]\to\T^n$ is the unique minimizing curve for $v(x)$. 

Suppose that there exists a minimizing measure $\mu$ of (M) for which we have a point $(x^\ast,\xi^\ast)\in \supp(\mu)\setminus \tilde{\mathcal{M}}_\alpha(c)$.  Then, $(x^\ast,\xi^\ast)$ cannot be any $\alpha$-limit point of $\phi^s_L(x^\ast,\xi^\ast)$. Therefore, we have a closed $\delta$-ball $B^\ast_\delta$ of $(x^\ast,\xi^\ast)$ to which $\phi^s_L(x^\ast,\xi^\ast)$ never comes back for $s\to-\infty$. The following two cases are possible: 
\begin{enumerate}
\item[(a)] Taking smaller $\delta>0$ if necessary, the set $B^\ast_\delta\setminus\{ \phi^s_L(x^\ast,\xi^\ast) \}_{s\in\R}$ contains no other points of $\supp(\mu)$.
\item [(b)] There exists a sequence $(x_i,\xi_i)\in[B^\ast_\delta\cap\supp(\mu)]\setminus \{ \phi^s_L(x^\ast,\xi^\ast) \}_{s\in\R}$ such that $(x_i,\xi_i)\to(x^\ast,\xi^\ast)$ as $i\to\infty$.
\end{enumerate}
\indent 
In case (a), taking a continuous function $\varphi:\T^n\times\R^n\to\R$ such that $\supp (\varphi)\subset B^\ast_\delta$ and $\varphi>0$ in $\supp(\varphi)$, we have for each $T>0$,
\begin{align}
&\int_{\T^n\times\R^n}\left\{  \frac{1}{T}\int^0_{-T}\varphi(\phi^s_L(x,\xi)) ds \right\}\,d\mu= 
\frac{1}{T}\int^0_{-T} \left\{ \int_{\T^n\times\R^n}\varphi(\phi^s_L(x,\xi))d\mu\right\}\,ds
\label{12345}\\
=&\, 
\frac{1}{T}\int^0_{-T} \left\{ \int_{\T^n\times\R^n}\varphi(\phi^0_L(x,\xi))d\mu\right\}\,ds
=
\int_{\T^n\times\R^n}\varphi(x,\xi)d\mu>0. 
\nonumber
\end{align}
Since $\phi^s_L(x,\xi)$ with $(x,\xi)\in\supp(\mu)\cap B^\ast_\delta$ never comes back to $B^\ast_\delta$ for $s\to-\infty$, we have $ \frac{1}{T}\int^0_{-T}\varphi(\phi^s_L(x,\xi)) ds \to0$ as $T\to\infty$ pointwise on $\supp(\mu)\cap B^\ast_\delta$ and hence, the left-hand-side of (\ref{12345}) is to be $0$, which is a contradiction. 

In case (b), we cannot find any subsequence of $(x_i,\xi_i)$, still denoted by  $(x_i,\xi_i)$, such that  $(x_i,\xi_i)\in \tilde{\mathcal{M}}_\alpha(c)$, because otherwise $(x^\ast,\xi^\ast)$ must belong to $\tilde{\mathcal{M}}_\alpha(c)$. Hence, taking smaller $\delta>0$ if necessary, all points of  $[B^\ast_\delta\cap\supp(\mu)]\setminus \{ \phi^s_L(x^\ast,\xi^\ast) \}_{s\in\R}$ do not belong to $\tilde{\mathcal{M}}_\alpha(c)$. Note that if $(x,\xi)\in\supp(\mu)$, the $\alpha$-limit points of $\phi^s_L(x,\xi)$ belong to $\supp(\mu)$.  Therefore, for each $(x,\xi)\in  [B^\ast_\delta\cap\supp(\mu)]\setminus \{ \phi^s_L(x^\ast,\xi^\ast) \}_{s\in\R}$, the $\alpha$-limit points of $\phi^s_L(x,\xi)$  cannot belong to $[B^\ast_\delta\cap\supp(\mu)]\setminus \{ \phi^s_L(x^\ast,\xi^\ast) \}_{s\in\R}$, nor to  $B^\ast_\delta\cap \{ \phi^s_L(x^\ast,\xi^\ast) \}_{s\in\R}$ because otherwise  $\{ \phi^s_L(x^\ast,\xi^\ast) \}_{s\in\R} \subset \tilde{\mathcal{M}}_\alpha(c)$ due to (iii) of Theorem \ref{invariance}. Thus, since $\phi^s_L(x,\xi)$ with $(x,\xi)\in\supp(\mu)\cap B^\ast_\delta$ never comes back to $B^\ast_\delta$ for $s\to-\infty$, the same argument with (\ref{12345}) as that in the case of (a) yields a contradiction.

We prove (iii). It is enough to show  $\mathcal{M}_\alpha(c )\subset\mathcal{A}(c )$, because if $x\in \mathcal{M}_\alpha(c )\cap\mathcal{A}(c )$, we have a viscosity solution $v$ for which $(x,\xi)=(x,H_p(x,c+v_x(x)))\in\tilde{\mathcal{M}}_\alpha(v;c )$ 
and hence $(x,\xi)\in\tilde{\mathcal{A}}(c )$ due to (\ref{mather-deri}) on $\tilde{\mathcal{A}}(c )$ for any viscosity solution.  
Let $x$ be an arbitrary point of $\mathcal{M}_\alpha(c )$ for which we have a minimizing curve $\gamma:(-\infty,0]\to\T^n$ and $T_i\to\infty$ such that $\gamma(-T_i)\to x$ as $i\to\infty$. Then, it holds that 
\begin{eqnarray*}
v(\gamma(-T_i))&=&\int^{-T_i}_{-T_{i+j}}L(\gamma(s),\gamma'(s))-c\cdot\gamma'(s) +h(c )\, ds +v(\gamma(-T_{i+j}))
\end{eqnarray*}
or, with $\gamma_i(s):=\gamma(s-T_i):[-T_{i+j}+T_i,0]\to\T^n$,
\begin{eqnarray*}
v(\gamma_i(0))&=&\int^0_{-T_{i+j}+T_i}L(\gamma_i(s),\gamma_i'(s))-c\cdot\gamma_i'(s)+h(c )\, ds+v(\gamma_i(-T_{i+j}+T_i)).
\end{eqnarray*}
Define 
\begin{eqnarray*}
\gamma_{ij}(s):=\left\{
\begin{array}{l}
\gamma_i(s)\,\,\,\,\,\,\,\text{for} \ s\in[-T_{i+j}+T_i,0],  \\
\dis\gamma_i(-T_{i+j}+T_i)+\frac{\gamma_i(-T_{i+j}+T_i)-x}{\ep_{ij}} \{s-(-T_{i+j}+T_i)\,\\
\qquad\qquad\qquad\qquad\qquad\qquad\,\,\,\,\,\,\,\text{for}s\in[-T_{i+j}+T_i-\ep_{ij},-T_{i+j}+T_i],
\end{array}
\right.
\end{eqnarray*}
where $\ep_{ij}:=\sqrt{|\gamma_i(-T_{i+j}+T_i)-x|}$. Then we see that 
\begin{eqnarray*}
\mathfrak{h}^c(\gamma_i(0),x)&\le& \liminf_{j\to\infty} \int^0_{-T_{i+j}+T_i-\ep_{ij}}L(\gamma_{ij}(s),\gamma_{ij}'(s))-c\cdot\gamma_{ij}'(s)+h(c )\, ds\\
&=&\liminf_{j\to\infty} v(\gamma_i(0))-v(\gamma_i(-T_{i+j}+T_i))+O(\ep_{ij})
=v(\gamma(-T_i))-v(x).
\end{eqnarray*}
Since $\gamma(-T_i)\to x$ as $i\to\infty$ and continuity of $\mathfrak{h}^c$, we have $\mathfrak{h}^c(x,x)\le0$.
On the other hand, for a minimizing curve $\gamma^\ast$ of the variational problem 
$$\inf_{\gamma\in AC,\gamma(0)=x,\gamma(-T)=x}\int^0_{-T}L(\gamma(s),\gamma'(s))-c\cdot\gamma'(s)+h(c ) \,ds,$$
we have 
$$ \int^0_{-T}L(\gamma^\ast(s),\gamma^\ast{}'(s))-c\cdot\gamma^\ast{}'(s)+h(c ) \,ds\ge v(\gamma^\ast(0))-v(\gamma^\ast(-T))=0$$
for each $T>0$, which implies $\mathfrak{h}^c(x,x)\ge0$. Hence, we obtain $\mathfrak{h}^c(x,x)=0$ and therefore $x\in \mathcal{A}(c)$. Continuity of $\mathfrak{h}^c$ implies $\mathfrak{h}^c(x,x)=0$ also for accumulating points of $\mathcal{M}_\alpha(c)$. Thus, we conclude $\mathcal{M}_\alpha(c)\subset \mathcal{A}(c)$, $\tilde{\mathcal{M}}_\alpha(c)\subset \tilde{\mathcal{A}}(c)$.
\end{proof}
%%%%%%%%%
Since any viscosity solution of (\ref{HJ}) is differentiable on the projected Aubry set, we have
%%%%%%%%%%%
\begin{Cor}\label{2626}
 Any viscosity solution $v$ of (\ref{HJ}) is differentiable on $\mathcal{M}_\alpha(c)$ with the common derivative $v_x(x)=L_\xi(x,\xi)-c$ for each point $(x,\xi)\in\tilde{\mathcal{M}}_\alpha(c)$.
\end{Cor} 
%%%%%%%%%%
\noindent This is apparently not trivial from the viscosity-solution-wise definition of $\mathcal{M}_\alpha(c)$, $\tilde{\mathcal{M}}_\alpha(c)$. 
%%%%%%%%%
\begin{Rem}\label{2727}
\begin{enumerate}[{\rm(i)}]
\item  $\mathcal{M}_\alpha(c )$ can be strictly smaller than $\mathcal{A}(c )$ in general. 
\item Currently it is not clear if $\mathcal{M}(c )$ is  strictly smaller than $\mathcal{M}_\alpha(c )$ in general or not.
\end{enumerate}
\end{Rem}
%%%%%%%%%%
\noindent In fact, (i) of Remark \ref{2727} is the case, when there are homoclinic/heteroclinic orbits. For instance, suppose that Hamiltonian dynamics with $n=1$ has a single stationary solution $x_0$ with a homoclinic orbit $\gamma^\ast$ moving over $\T\setminus\{x_0\}$. There exists a value $c$ such that $\graph(c+v_x)=\{ (\gamma^\ast(s),c+v_x(\gamma^\ast(s))) \}_{s\in\R}$. Then, $\mathcal{A}(c )$ is equal to  $\{\gamma^\ast(s)\}_{s\in\R}\cup\{x_0\}=\T$, and $\mathcal{M}_\alpha(c )$ is equal to $\{x_0\}$. 

In regards to (ii) of Remark \ref{2727}, one could expect the case where $\mathcal{M}(c )$ is strictly smaller than $\mathcal{M}_\alpha(c )$ because of the following observation. Let $\gamma^\ast(s)$ be a minimizing curve and $(x^\ast,\xi^\ast)$ be an $\alpha$-limit point of $(\gamma^\ast(s),\gamma^\ast{}'(s))$. Let $\tau(T)$ denote the total length of time for which $(\gamma^\ast(s),\gamma^\ast{}'(s))$ stays in a $\delta$-ball of  $(x^\ast,\xi^\ast)$ within $s\in[-T,0]$. It seems that $r(\delta):=\lim_{T\to\infty}\tau(T)/T$ determines whether or not $(x^\ast,\xi^\ast)$ is a point of the support of the minimizing measure induced by (\ref{holo110}) and (\ref{holo11}), namely, $(x^\ast,\xi^\ast)$ would not belong to the support,  if $r(\delta)=0$ for some $\delta>0$.     
%%%%%%%%%%%%%%%%%%%%%%%%%%%%
%%%%%%%%%%%%%%%%%%%%%%%%%%%%
%%%%%%%%%%%%%%%%%%%%%%%%%%%%
%%%%%%%%%%%%%%%%%%%%%%%%%%%%
\setcounter{section}{2}
\setcounter{equation}{0}
\section{Weak KAM theory for discount problem}
We extend what we observed in Section 2 to the discount problem. The family of $\alpha$-limit points of minimizing curves has properties similar  to those of $\mathcal{M}_\alpha(c)$, $\tilde{\mathcal{M}}_\alpha(c)$. We present certain minimizing measures to relate them to the family of $\alpha$-limit points.  These minimizing measures are different from the ones used in \cite{Davini}, \cite{MT} for the selection criterion in the vanishing discount process.   
%%%%%%%%%%%%%
\subsection{Viscosity solution and minimizing curve}  
Let $v^\ep$ be the unique viscosity solution of (\ref{dHJ}). Note that $v^\ep$ is Lipschitz continuous and semiconcave. It is well-known that $v^\ep$ satisfies for each $x\in\T^n$ and $T>0$,
\begin{eqnarray}\label{d-value-func}
v^\ep(x)&=&\inf_{\gamma\in AC,\gamma(0)=x} \Big\{\int^0_{-T}e^{\ep s}  (L(\gamma(s),\gamma'(s))-c\cdot\gamma'(s)+h(c )  )ds \\\nonumber 
&&+ e^{-\ep T}v^\ep(\gamma(-T))\Big\}.
\end{eqnarray} 
It follows from Tonelli's theory that there exists a minimizing curve $\gamma^\ep:[-T,0]\to\T^n$ for $v^\ep(x)$, which is a $C^2$-solution of the Euler-Lagrange equation generated by $L^{c,\ep}:=e^{\ep s}\{L-c\cdot\xi+h(c )\}$,
\begin{eqnarray} \label{dEL}
&&\frac{d}{ds} \{ L^{c,\ep}_\xi(x(s),x'(s)) \}= L^{c,\ep}_x(x(s),x'(s))\\\nonumber
\Leftrightarrow
&& \frac{d}{ds}\{ L_\xi(x(s),x'(s)) \}=L_x(x(s),x'(s))-\ep L_\xi(x(s),x'(s))+\ep c. 
\end{eqnarray}
The flow of (\ref{dEL}) is denoted by $\phi^s_{L,c,\ep}$, i.e., $\phi^s_{L,c,\ep}(x(0),x'(0))=(x(s),x'(s))$. Set $p^\ep(s):=L_\xi(\gamma^\ep(s),\gamma^\ep{}'(s))$. Then $(\gamma^\ep(s),p^\ep(s))$ solves the equation,
\begin{eqnarray}\label{dHS}
\left\{
\begin{array}{l}
x'(s)=H_p(x(s),p(s)),\\
p'(s)=-H_x(x(s),p(s))+\ep c- \ep p(s). 
\end{array}
\right. 
\end{eqnarray}
%\begin{eqnarray}\label{dHS}
%x'(s)=H_p(x(s),p(s)),\quad p'(s)=-H_x(x(s),p(s))+\ep c- \ep p(s).
%\end{eqnarray}
We see that (\ref{dEL}), (\ref{dHS}) can be regarded as (\ref{EL}), (\ref{HS}) with a friction term.  The flow of (\ref{dHS}) is denoted by $\phi^s_{H,c,\ep}$.
%%%%%%%%%
\begin{Prop}\label{programing}
Let $\gamma^\ep:[-T,0]\to\T^n$ be a minimizing curve for $v^\ep(x)$. For each $\tau\in[0,T)$,  we have 
\begin{align*}
&v^\ep(\gamma^\ep(-\tau))\\
=&\, e^{\ep \tau}\Big\{\int^{-\tau}_{-T}e^{\ep s}  (L(\gamma^\ep(s),\gamma^\ep{}'(s))-c\cdot\gamma^\ep{}'(s)+h(c )  )ds 
+ e^{-\ep T}v^\ep(\gamma^\ep(-T))\Big\}\\
=&\,\int^0_{-T+\tau}e^{\ep s}  (L(\gamma^\ep(s-\tau),\gamma^\ep{}'(s-\tau))-c\cdot\gamma^\ep{}'(s-\tau)+h(c )  )ds 
+ e^{\ep (-T+\tau)}v^\ep(\gamma^\ep(-T)).
\end{align*} 
%\begin{eqnarray*}
%v^\ep(\gamma^\ep(-\tau))&=&e^{\ep \tau}\Big\{\int^{-\tau}_{-T}e^{\ep s}  (L(\gamma^\ep(s),\gamma^\ep{}'(s))-c\cdot\gamma^\ep{}'(s)+h(c )  )ds \\
%&&+ e^{-\ep T}v^\ep(\gamma^\ep(-T))\Big\}\\
%&=&\int^0_{-T+\tau}e^{\ep s}  (L(\gamma^\ep(s-\tau),\gamma^\ep{}'(s-\tau))-c\cdot\gamma^\ep{}'(s-\tau)+h(c )  )ds \\\nonumber
%&&+ e^{\ep (-T+\tau)}v^\ep(\gamma^\ep(-T)).
%\end{eqnarray*} 
\end{Prop}
%%%%%%%%
\noindent This is well-known as a dynamic programing principle for the discount value function (\ref{d-value-func}).
It follows from Proposition \ref{programing} that  we have  
\begin{eqnarray}\label{value6}
v^\ep(x)&=&\int^0_{-\tau}e^{\ep s} (L(\gamma^\ep(s),\gamma^\ep{}'(s))-c\cdot\gamma^\ep{}'(s)+h(c )  )ds\\\nonumber
&&\qquad\qquad\qquad +e^{-\ep \tau}v^\ep(\gamma^\ep(-\tau))\mbox{\,\,\,\, for each $\tau\in[0,T]$}.
\end{eqnarray}
%%%%%%%%%%%%%
\begin{Prop}\label{differentiable}
Let $\gamma^\ep:[-T,0]\to\T^n$ be a minimizing curve for $v^\ep(x)$. Then $v^\ep$ is differentiable on $\gamma^\ep(s)$ for all $s\in[-T,0)$ satisfying  
$$v^\ep_x(\gamma^\ep(s))=L_\xi(\gamma^\ep(s),\gamma^\ep{}'(s))-c.$$
If $v^\ep_x(x)$ exists, this holds for $s=0$ and  $\gamma^\ep$ is the unique minimizing curve for $v^\ep(x)$.
\end{Prop}
%%%%%%%%%%%%%%%
\begin{proof}
Since $v^\ep$ is semiconcave, the superdifferential is a non-empty set: $D^+v^\ep(\gamma^\ep(s))\neq\emptyset$ for all $s\in[-T,0]$ (see, e.g.,  \cite{Cannarsa}). Hence it is enough to check that the subdifferential contains $L_\xi(\gamma^\ep(s),\gamma^\ep{}'(s))-c$.    
For any $\tau\in[-T,0)$, we have 
\begin{eqnarray*}
v^\ep(x)&=&\int^0_{-T} e^{\ep s}(L(\gamma^\ep(s), \gamma^\ep{}'(s))-c\cdot\gamma^\ep{}'(s) +h(c ))ds+e^{-\ep T}v^\ep(\gamma^\ep(-T))\\
&=&\int^0_\tau e^{\ep s}(L(\gamma^\ep(s), \gamma^\ep{}'(s))-c\cdot\gamma^\ep{}'(s) + h(c ))ds+e^{\ep \tau}v^\ep(\gamma^\ep(\tau)).
\end{eqnarray*}
Take $\delta>0$ small so that $\tau+\delta\in(-T,0)$. Introduce a continuous curve $\gamma$ with $w\in\R^n$ as
\begin{eqnarray*}
\gamma(s):=\left\{
\begin{array}{lll}
\gamma^\ep(s) &\text{for} \ \tau+\delta\le s\le 0, \\
\dis \gamma^\ep(s)+(\tau+\delta-s)\frac{w}{\delta} &\text{for} \ \tau\le s\le\tau+\delta.
\end{array}
\right.
\end{eqnarray*}
Due to the variational property, we have 
\begin{eqnarray*}
v^\ep(x)\le \int^0_\tau e^{\ep s}(L(\gamma(s), \gamma'(s))-c\cdot\gamma'(s) +h(c ))ds+e^{\ep \tau}v^\ep(\gamma(\tau)).
\end{eqnarray*}
Hence, we obtain
\begin{eqnarray*}
0&\le&\int^{\tau+\delta}_\tau e^{\ep s}\Big[\big\{ L(\gamma^\ep(s)+(\tau+\delta-s)\frac{w}{\delta}, \gamma^\ep{}'(s)-\frac{w}{\delta})-c\cdot(\gamma^\ep{}'(s)-\frac{w}{\delta})\big\} \\
&&- \big\{L(\gamma^\ep(s), \gamma^\ep{}'(s))-c\cdot\gamma^\ep{}'(s)    \big\}\Big ]ds+e^{\ep\tau}\big\{ v^\ep(\gamma^\ep(\tau)+w)-v^\ep(\gamma^\ep(\tau)) \big\}\\
&=&e^{\ep (\tau+\theta\delta)}\Big[\big\{ L(\gamma^\ep(\tau+\theta\delta)+(\delta-\theta\delta)\frac{w}{\delta}, \gamma^\ep{}'(\tau+\theta\delta)-\frac{w}{\delta})-c\cdot(\gamma^\ep{}'(\tau+\theta\delta)-\frac{w}{\delta})\big\} \\
&&- \big\{L(\gamma^\ep(\tau+\theta\delta), \gamma^\ep{}'(\tau+\theta\delta))-c\cdot\gamma^\ep{}'(\tau+\theta\delta)    \big\}\Big ]\delta\\
&&+e^{\ep\tau}\big\{ v^\ep(\gamma^\ep(\tau)+w)-v^\ep(\gamma^\ep(\tau)) \big\}\\
&=&e^{\ep (\tau+\theta\delta)}\Big\{ L_x(\gamma^\ep(\tau+\theta\delta)+\tilde{\theta}(1-\theta)w, \gamma^\ep{}'(\tau+\theta\delta)-\tilde{\theta}\frac{w}{\delta})\cdot(1-\theta)w\\
&&+L_\xi(\gamma^\ep(\tau+\theta\delta)+\tilde{\theta}(1-\theta)w, \gamma^\ep{}'(\tau+\theta\delta)-\tilde{\theta}\frac{w}{\delta})\cdot(-\frac{w}{\delta})+c\cdot\frac{w}{\delta}\Big\}\delta\\
&&+e^{\ep\tau}\big\{ v^\ep(\gamma^\ep(\tau)+w)-v^\ep(\gamma^\ep(\tau)) \big\},
\end{eqnarray*}
for some $\theta,\tilde{\theta}\in(0,1)$. Therefore, we have 
\begin{eqnarray*}
&&\frac{v^\ep(\gamma^\ep(\tau)+w)-v^\ep(\gamma^\ep(\tau)) - \{L_\xi(\gamma^\ep(\tau),\gamma^\ep{}'(\tau))-c\}\cdot w}{|w|}  \\
&&\ge -e^{\ep \theta\delta} L_x(\gamma^\ep(\tau+\theta\delta)+\tilde{\theta}(1-\theta)w, \gamma^\ep{}'(\tau+\theta\delta)-\tilde{\theta}\frac{w}{\delta})\cdot(1-\theta)\frac{w}{|w|}\delta\\
&&+e^{\ep \theta\delta} L_\xi(\gamma^\ep(\tau+\theta\delta)+\tilde{\theta}(1-\theta)w, \gamma^\ep{}'(\tau+\theta\delta)-\tilde{\theta}\frac{w}{\delta})\cdot \frac{w}{|w|}
-L_\xi(\gamma^\ep(\tau),\gamma^\ep{}'(\tau))\cdot \frac{w}{|w|}\\
&&+(1-e^{\ep \theta\delta})c\cdot \frac{w}{|w|}.
\end{eqnarray*}
Taking $\delta=\sqrt{|w|}$ and sending $w$ to $0$, we conclude that 
$$\liminf_{|w|\to0}\frac{v^\ep(\gamma^\ep(\tau)+w)-v^\ep(\gamma^\ep(\tau)) - \{L_\xi(\gamma^\ep(\tau),\gamma^\ep{}'(\tau))-c\}\cdot w}{|w|}  \ge0,$$
which means that the subdifferential $D^-v^\ep (\gamma^\ep(\tau))$ contains $L_\xi(\gamma^\ep(\tau),\gamma^\ep{}'(\tau))-c$ for all $\tau\in[-T,0)$.

Similar reasoning with the curve 
\begin{eqnarray*}
\gamma(s):=\left\{
\begin{array}{lll}
\displaystyle
\gamma^\ep(s)+\frac{s+\delta}{\delta}w &\text{for} \ -\delta\le s\le 0, \\
\dis \gamma^\ep(s) &\text{for} \ -T\le s\le-\delta
\end{array}
\right.
\end{eqnarray*}
shows that $D^+v^\ep(\gamma^\ep(0))$ contains $L_\xi(\gamma^\ep(0),\gamma^\ep{}'(0))-c$. If $v_x^\ep(\gamma^\ep(0))=v^\ep_x(x)$ exists, the set $D^+v^\ep(\gamma^\ep(0))$ must be singleton and therefore $v_x^\ep(\gamma^\ep(0))=L_\xi(\gamma^\ep(0),\gamma^\ep{}'(0))-c$.  Uniqueness of the discount Euler-Lagrange equation (\ref{dEL}) yields uniqueness of the minimizing curve.  
\end{proof}
%%%%%%%%%%%%%%%
\noindent Each minimizing curve $\gamma^\ep: [-T,0]\to\T^n$ for $v^\ep(x)$ can be extended to $(-\infty,0]$. We call an extended minimizing curve defined on $(-\infty,0]$ a one-sided global minimizing curve or just minimizing curve. Since $v^\ep_x(\gamma^\ep(-T))=L_\xi(\gamma^\ep(-T),\gamma^\ep{}'(-T))-c$ due to Proposition \ref{differentiable}, the minimizing curve for $v(\gamma^\ep(-T))$ is uniquely obtained as $\gamma^\ep(-T+\cdot)$ on any interval $[-\tilde{T},0]$. Therefore, with (\ref{value6}), we conclude that for a one-sided global minimizing curve $\gamma^\ep$ for $v^\ep(x)$ we have
\begin{eqnarray}\label{ddd-value}
v^\ep(x)&=&\int^0_{-\tau}e^{\ep s} (L(\gamma^\ep(s),\gamma^\ep{}'(s))-c\cdot\gamma^\ep{}'(s)+h(c )  )ds\\\nonumber
&&\qquad\qquad\qquad\qquad\qquad\qquad+e^{-\ep \tau}v^\ep(\gamma^\ep(-\tau))\mbox{ for any $\tau\ge0$.}
\end{eqnarray}
Equivalence between (\ref{dEL}) and (\ref{dHS}), and Proposition \ref{differentiable} yield  
%%%%%%%%%%
\begin{Thm}
The set $\graph(c+v^\ep_x)$ is backward invariant under the discount Hamiltonian flow $\phi^s_{H,c,\ep}$, i.e., $\phi^s_{H,c,\ep}(\graph(c+v^\ep_x))\subset \graph(c+v^\ep_x)\mbox{\quad for all  $s\le0$}$. 
\end{Thm}
%%%%%%%%%%%%%%
\subsection{$\alpha$-limit point}
 Now we introduce the family of $\alpha$-limit points of all one-sided global minimizing curves of $v^\ep$:
$$\mathcal{M}_\alpha^\ep(c ):=\overline{\{ x^\ep\,|\,\mbox{$x^\ep$: $\alpha$-limit point of $\gamma^\ep$, $\gamma^\ep$: minimizer for $v^\ep(x)$, $x\in\T^n$ } \}}.$$
Here are properties of $\mathcal{M}^\ep_\alpha(c )$:
%%%%%%%%
\begin{Thm}\label{derivative8}
\begin{enumerate}[{\rm(i)}]
\item The viscosity solution $v^\ep$ of \eqref{dHJ} is differentiable on $\mathcal{M}_\alpha^\ep(c)$. 
\item Let $\gamma^\ep$ be a minimizing curve of $v^\ep$ and $(x^\ep,\xi^\ep)$ be an $\alpha$-limit point of $(\gamma^\ep(s),\gamma^\ep{}'(s))$. Then $\xi^\ep=H_p(x^\ep,c+v_x^\ep(x^\ep))$.
\item Let $x^\ep$ be a point of $\mathcal{M}_\alpha^\ep(c)$ and $x(s)$ be the solution of the discount Euler-Lagrange equation \eqref{dEL} with $x(0)=x^\ep$, $x'(0)=H_p(x^\ep,c+v_x^\ep(x^\ep))$. Then $x(s)\in\mathcal{M}_\alpha^\ep(c)$ and $v^\ep_x(x(s))=L_\xi(x(s),x'(s))-c$ for all $s\in\R$. 
\end{enumerate}
\end{Thm}
%%%%%%%%
\begin{proof}
Let $x^\ep$ be an arbitrary point of $\mathcal{M}_\alpha^\ep(c)$ for which  there exists a minimizing curve $\gamma^\ep:(-\infty,0]\to\T^n$ and $\tau_j\to-\infty$ ($j\to\infty$) such that $\gamma^\ep(\tau_j)\to x^\ep$ ($j\to\infty$). Since 
$$\gamma^\ep{}'(\tau_j)=H_p(\gamma^\ep(\tau_j),c+v^\ep_x(\gamma^\ep(\tau_j)))\mbox{ ($\Leftrightarrow$ $v^\ep_x(\gamma^\ep(\tau_j))=L_\xi(\gamma^\ep(\tau_j),\gamma^\ep{}'(\tau_j))-c$)}$$
is a bounded sequence, there exists a subsequence, still denoted by $\{\gamma^\ep{}'(\tau_j)\}_j$, which converges to some $\xi^\ep\in\R^n$. Let $x(s):\R\to\T^n$ be the  solution of (\ref{dEL}) with the initial condition $x(0)=x^\ep,\,x'(0)=\xi^\ep$. It follows from the continuous dependence on initial conditions that we have for any $\alpha>0$, 
$$(\gamma^\ep(\cdot+\tau_j), \gamma^\ep{}'(\cdot+\tau_j))\to (x(\cdot),x'(\cdot))\mbox{ uniformly on $[-\alpha,\alpha]$ ($j\to\infty$)}.$$   
Hence, we see that $x(s)\in\mathcal{M}_\alpha^\ep(c)$ for all $s\in\R$ by the definition of $\mathcal{M}_\alpha^\ep(c)$. Since $\gamma^\ep(\cdot+\alpha+\tau_j):[-2\alpha,0]\to\T^n$ is the unique minimizing curve for $v(\gamma^\ep(\alpha+\tau_j))$ for each fixed large $j$, i.e.,
\begin{eqnarray*}
v^\ep(\gamma^\ep(\alpha+\tau_j))&=&\int^0_{-2\alpha}e^{\ep s}(L(\gamma^\ep(s+\alpha+\tau_j),\gamma^\ep{}'(s+\alpha+\tau_j))-c\cdot\gamma^\ep{}'(s+\alpha+\tau_j)\\
&& +h(c ))ds+e^{-\ep\cdot 2\alpha}v^\ep(\gamma^\ep(-\alpha+\tau_j)).
\end{eqnarray*}
We obtain by letting $j\to\infty$,
\begin{eqnarray*}
v^\ep(x(\alpha))&=&\int^0_{-2\alpha}e^{\ep s}(L(x(\alpha+s),x'(\alpha+s))-c\cdot x'(\alpha+s) +h(c ))ds\\
&&+e^{-\ep\cdot2\alpha}v^\ep(x(-\alpha)). 
\end{eqnarray*}  
Therefore, we see that $x(\alpha+s)$, $s\in[-2\alpha,0]$ is a minimizing curve for $v^\ep(x(\alpha))$, and hence that $v_x^\ep(x(\alpha+s))=L_\xi(x(\alpha+s),x'(\alpha+s))-c$ for all $s\in[-2\alpha,0)$. For $s=-\alpha$, we obtain $v_x^\ep(x^\ep)=L_\xi(x^\ep,\xi^\ep)-c$ and $\xi^\ep=x'(0)=H_p(x^\ep,c+v_x(x^\ep))$. Continuity yields the same for each accumulating point $(x^\ep,\xi^\ep)$ of $\mathcal{M}_\alpha^\ep(c)$. Noting that $\alpha>0$ is arbitrary, we complete the proof.  \end{proof}
%%%%%%%%%%%%
\noindent The following sets are well-defined:
\begin{eqnarray*}
\tilde{\mathcal{M}}_\alpha^\ep(c )&:=&\{(x^\ep,H_p(x^\ep,c+v^\ep_x(x^\ep)))\,|\,x^\ep\in\mathcal{M}_\alpha^\ep(c ) \},\\
\tilde{\mathcal{M}}_\alpha^\ep{}^\ast(c )&:=&\{(x^\ep,c+v^\ep_x(x^\ep))\,|\,x^\ep\in\mathcal{M}_\alpha^\ep(c ) \}.
\end{eqnarray*}
%%%%%%%%
\begin{Rem}
 The set $\mathcal{M}_\alpha^\ep(c)$ is the closure of the family of all the curves $x(s)$ in Theorem \ref{derivative8}, and 
\begin{eqnarray*}
\tilde{\mathcal{M}}_\alpha^\ep(c)&=&\overline{\{ (x^\ep,\xi^\ep)\,|\, \mbox{$(x^\ep,\xi^\ep)$: $\alpha$-limit point of $(\gamma^\ep(s),\gamma^\ep{}'(s))$}},\\
&&\qquad\qquad\qquad\qquad\qquad\qquad \overline{\mbox{$\gamma^\ep$: minimizing curve of $v^\ep$ } \}}.
\end{eqnarray*} 
\end{Rem}
%%%%%%%
\noindent Here are properties of $\tilde{\mathcal{M}}^\ep_\alpha(c)$ and $\tilde{\mathcal{M}}^\ep{}^\ast_\alpha(c)$:
%%%%%%%%%
\begin{Thm}
\begin{enumerate}[{\rm(i)}]
\item For each point $x\in\T^n$, there exists $\xi\in\R^n$ such that $\phi^s_{L,c,\ep}(x,\xi)$ falls into $\tilde{\mathcal{M}}_\alpha^\ep(c)$ as $s\to-\infty$, i.e., any accumulating point of $\{\phi^s_{L,c,\ep}(x,\xi)\}_{s\le0}$ belongs to $\tilde{\mathcal{M}}_\alpha^\ep(c)$. If $v^\ep_x(x)$ exists, $\xi=H_p(x,c+v_x^\ep(x))$.
\item  $\tilde{\mathcal{M}}_\alpha^\ep(c)$ is a $\phi^s_{L,c,\ep}$-invariant subset of $\T^n\times\R^n$, i.e., $\phi^s_{L,c,\ep}(\tilde{\mathcal{M}}_\alpha^\ep(c))=\tilde{\mathcal{M}}_\alpha^\ep(c)$  for all $s\in\R$. 
\item For each point $x\in\T^n$,  there exists $p\in\R^n$ such that $\phi^s_{H,c,\ep}(x,p)$ falls into $\tilde{\mathcal{M}}^\ep_\alpha{}^\ast(c)$ as $s\to-\infty$ along $\graph(c+v_x^\ep)$. If $v^\ep_x(x)$ exists, $p=c+v_x^\ep(x)$.
\item $\tilde{\mathcal{M}}^\ep_\alpha{}^\ast(c)$ is a $\phi^s_{H,c,\ep}$-invariant subset of $\graph(c+v^\ep_x)$. 
\end{enumerate}
\end{Thm}
%%%%%%%%
\noindent Proof is the same as that of Theorem \ref{similar}.
%%%%%%%
%\begin{proof} 
%1. is clear by definition. 2. is clear by 3. of Theorem \ref{derivative8}. Set $p:=L_\xi(x,\xi)$, $\xi:=\gamma^\ep{}'(0)$ with a minimizing curve $\gamma^\ep$ for $v^\ep(x)$. Then we have $\phi^s_{H,c,\ep}(x,p)=(\pi\circ \phi^s_{L,c,\ep}(x,\xi), L_\xi(\phi^s_{L,c,\ep}(x,\xi)))$. Hence, 3. follows from above 1., Proposition \ref{differentiable} with equivalence between (\ref{dEL}) and (\ref{dHS}), and 2. of Theorem \ref{derivative8}. Let $(x^\ep,p^\ep)$ be an arbitrary point of $\tilde{\mathcal{M}}_\alpha^\ep{}^\ast(c)$. 4. follows from  3. of Theorem \ref{derivative8} with $x'(0)=H_p(x^\ep,p^\ep)$ and equivalence between (\ref{EL}) and (\ref{HS})
%\end{proof}
%%%%%%%%%%%%%%%%%%
%%%%%%%%%%%%%%%%%%
\subsection{Two types of minimizing measure}
We observe measure-theoretical characterization of the discount problem and show relation between $\mathcal{M}_\alpha^\ep(c )$, $\tilde{\mathcal{M}}_\alpha^\ep(c )$ and minimizing measures. There are two types of minimizing measures obtained from (M1)$^\ep$ and (M2)$^\ep$ in Introduction. As is observed below, (M1)$^\ep$ is a direct generalization of (M), and is well-related to $\mathcal{M}_\alpha^\ep(c )$, $\tilde{\mathcal{M}}_\alpha^\ep(c )$. On the other hand,  (M2)$^\ep$ is obtained naturally from (\ref{d-value-func}) \cite{Davini}, or the nonlinear adjoint method as a solution of the adjoint PDE  \cite{MT}.  We will see that the minimizing measures of (M2)$^\ep$ are different from those of (M1)$^\ep$, namely the minimizing measures of (M2)$^\ep$ are not $\phi^s_{L,c,\ep}$-invariant in general and their support is not related to $\mathcal{M}_\alpha^\ep(c )$,  $\tilde{\mathcal{M}}_\alpha^\ep(c )$, whereas both tend to some of the minimizing measures of (M) as $\ep\to0+$. In the selection problem \cite{Davini}, \cite{MT}, the minimizing measures of (M2)$^\ep$ are exploited.  

Now we observe more on (M1)$^\ep$. Let $\mu$ be arbitrary $\phi^s_{L,c,\ep}$-invariant probability measure. Then we have for any $(x,\xi)\in\T^n\times\R^n$,
\begin{equation}\label{321-ineq}
v^\ep(\pi\circ\phi^0_{L,c,\ep}(x,\xi))\le\int^0_{-T}e^{\ep s} \{L^c(\phi^s_{L,c,\ep}(x,\xi))  + h(c) \}ds
+\ep^{-\ep T}v^\ep(\pi\circ\phi^{-T}_{L,c,\ep}(x,\xi)).
\end{equation}
%\begin{eqnarray}\label{321-ineq}
%v^\ep(\pi\circ\phi^0_{L,c,\ep}(x,\xi))&\le& \int^0_{-T}e^{\ep s} \{L^c(\phi^s_{L,c,\ep}(x,\xi))  + h(c) \}ds\\\nonumber
%&&\qquad\qquad\qquad\qquad+\ep^{-\ep T}v^\ep(\pi\circ\phi^{-T}_{L,c,\ep}(x,\xi)).
%\end{eqnarray}
Integrating the inequality with $\mu$ over $\T^n\times\R^n$, we obtain
\begin{eqnarray*}
\int_{\T^n\times\R^n}v^\ep(x)d\mu \le \frac{1-e^{-\ep T}}{\ep}\left\{ \int_{\T^n\times\R^n} L^c(x,\xi) \,d\mu + h(c)\right\}+e^{-\ep T}\int_{\T^n\times\R^n}v^\ep(x)d\mu.
\end{eqnarray*}
Letting $T\to\infty$, we have 
$$\int_{\T^n\times\R^n}L(x,\xi)-c\cdot\xi    -\ep v^\ep(x) \,d\mu\ge-h(c).$$
Let $\gamma^\ep(-\infty,0]\to\T^n$ be a minimizing curve of $v^\ep$. Note that $\{(\gamma^\ep(s),\gamma^\ep{}'(s))\}_{s\le0}$ is contained in a compact set $K\subset \T^n\times\R^n$ independent from $\ep$ and choice of $\gamma^\ep$, because $\{v^\ep\}_{\ep>0}$ is equi-Lipschitz continuous. Define the linear functional 
$$\Psi^\ep_T(\varphi):=\frac{1}{T}\int^0_{-T}\varphi(\gamma^\ep(s),\gamma^\ep{}'(s))ds,\,\,\,\,\,\varphi\in C_c(\T^n\times\R^n,\R).$$
Then, by the Riesz representation theorem, there exists a probability measure $\mu_T$ on $\T^n\times\R^n$ such that 
\begin{eqnarray}\label{d-holo11}
\Psi_T^\ep(\varphi)=\int_{\T^n\times\R^n}\varphi(x,\xi)d\mu_T\,\,\,\mbox{ for all $\varphi\in C_c(\T^n\times\R^n,\R)$}.
\end{eqnarray}
Since $\supp(\mu_T)$ is contained in the compact set $K\subset \T^n\times\R^n$ for all $T>0$ and $\ep>0$, we have $T_i\to\infty$ for which $\mu_{T_i}$ converges weakly to a probability measure  $\mu^\ep$. We see that $\mu^\ep$ is $\phi^s_{L,c,\ep}$-invariant. Furthermore, (\ref{d-holo11}) with $\varphi(x,\xi)=\ep v^\ep(x)$ (in $K$, otherwise re-defined to be $0$ continuously) and (\ref{dHJ}), as well as equivalence between (\ref{dEL}) and (\ref{dHS}) through the Legendre transform, yields 
\begin{eqnarray*}     
\int_{\T^n\times\R^n}\ep v^\ep(x) d\mu_{T_j}&=&\frac{1}{{T_j}}\int^0_{-{T_j}} \ep v^\ep(\gamma^\ep(s))ds\\
&=&\frac{1}{{T_j}}\int^0_{-{T_j}}-H(\gamma^\ep(s),c+v^\ep_x(\gamma^\ep(s)))+h(c )ds\\
&=&-\frac{1}{{T_j}}\int^0_{-{T_j}}\gamma^\ep{}'(s)\cdot \{c+v^\ep_x(\gamma^\ep(s))\}-L(\gamma^\ep(s),\gamma^\ep{}'(s))ds+h(c) \\
&=&-\frac{1}{{T_j}}\int^0_{-{T_j}}\frac{d}{ds}\{v^\ep(\gamma^\ep(s)) \}+\frac{1}{{T_j}}\int^0_{-{T_j}}L^c(\gamma^\ep(s),\gamma^\ep{}'(s))ds+h(c)\\
&=&-\frac{v^\ep(\gamma^\ep(0))-v^\ep(\gamma^\ep(-{T_j}))}{{T_j}}+\int_{\T^n\times\R^n}L^c(x,\xi)d\mu_{T_j}+h(c).
\end{eqnarray*}
Hence, we obtain with $j\to\infty$,
\begin{eqnarray}\label{3minimizing}     
\int_{\T^n\times\R^n}L(x,\xi)-c\cdot\xi-\ep v^\ep(x) \,\,d\mu^\ep=-h(c ),
\end{eqnarray}
namely, $\mu^\ep$ is a minimizing measure of (M1)$^\ep$. It is clear that $\mu^\ep$ satisfies the holonomic condition 
\begin{eqnarray*}
\int_{\T^n\times\R^n} \psi_x(x)\cdot\xi d\mu^\ep=0\mbox{ \,\,\,for all $\psi\in C^1(\T^n)$}.
\end{eqnarray*}
\indent We examine the support of minimizing measures of (M1)$^\ep$. Let $\mu^\ep$ be an arbitrary minimizing measure of (M1)$^\ep$. Since (\ref{321-ineq}) holds for any $(x,\xi)\in\T^n\times\R^n$, the minimizing property (\ref{3minimizing}) of $\mu^\ep$ implies that we have  
\[
v^\ep(\pi\circ\phi^0_{L,c,\ep}(x,\xi))=\int^0_{-T}e^{\ep s} \{L^c(\phi^0_{L,c,\ep}(x,\xi))  + h(c) \}ds
+\ep^{-\ep T}v^\ep(\pi\circ\phi^{-T}_{L,c,\ep}(x,\xi))
\]
for each $(x,\xi)\in\supp(\mu^\ep)$. 
%\begin{eqnarray*}
%v^\ep(\pi\circ\phi^0_{L,c,\ep}(x,\xi))&= &\int^0_{-T}e^{\ep s} \{L^c(\phi^0_{L,c,\ep}(x,\xi))  + h(c) \}ds\\
%&&\quad\qquad+\ep^{-\ep T}v^\ep(\pi\circ\phi^{-T}_{L,c,\ep}(x,\xi))\mbox{\quad for each $(x,\xi)\in\supp(\mu^\ep)$}.
%\end{eqnarray*}
Hence, we have $\xi=H_p(x,c+v_x^\ep(x))$ for each $(x,\xi)\in\supp(\mu^\ep)$, which means that $\supp(\mu^\ep)$ is contained in a compact set $K\subset\T^n\times\R^n$ for all $\ep>0$.  Since $\supp(\mu^\ep)\subset K$ for all $\ep>0$, there exists a weakly convergent subsequence $\mu^{\ep_j}$, $\ep_j\to0+$ ($j\to\infty$), whose limit is a minimizing measure of (M).  Define the set for each $\ep>0$,
$$\tilde{\mathcal{M}}^\ep(c ):= \overline{\bigcup_{\mu^\ep}\supp(\mu^\ep)},\,\,\,\,\mathcal{M}^\ep(c ):=\pi\tilde{\mathcal{M}}^\ep(c ),$$
where the union is taken with respect to all minimizing measures of (M1)$^\ep$. Reasoning similar to the proof of Theorem \ref{comparison} in Section 2 shows that each point $(x,\xi)\in\supp(\mu^\ep)$ itself is an $\alpha$-limit point of $\phi^s_{L,c,\ep}(x,\xi)$, where $\mu^\ep$ is a minimizing measure of (M1)$^\ep$. Therefore, we have the following inclusion:
%%%%%%%%%%%%
\begin{Thm}
${\mathcal{M}}^\ep(c )\subset{\mathcal{M}}_\alpha^\ep(c )\subset\T^n$, $\tilde{\mathcal{M}}^\ep(c )\subset\tilde{\mathcal{M}}_\alpha^\ep(c )\subset\T^n\times\R^n$.
\end{Thm}
%%%%%%%%%%%%
Now we observe more on (M2)$^\ep$. Let $x_0\in\T^n$ be fixed and $\gamma^\ep:(-\infty,0]\to\T^n$ be a minimizing curve for $v^\ep(x_0)$. Define the linear functional 
$$\tilde{\Psi}^\ep_T(\varphi):=\frac{\ep}{1-e^{-\ep T}}\int^0_{-T}e^{\ep s}\varphi(\gamma^\ep(s),\gamma^\ep{}'(s))ds,\,\,\,\,\,\varphi\in C_c(\T^n\times\R^n,\R).$$
Then, by  the Riesz representation theorem, there exists a probability measure $\tilde{\mu}_T$ on $\T^n\times\R^n$ such that 
\begin{eqnarray}\label{d2-holo11}
\tilde{\Psi}_T^\ep(\varphi)=\int_{\T^n\times\R^n}\varphi(x,\xi)d\tilde{\mu}_T\,\,\,\mbox{ for all $\varphi\in C_c(\T^n\times\R^n,\R)$}.
\end{eqnarray}
Since $\supp(\tilde{\mu}_T)$ is contained in a compact set $K\subset \T^n\times\R^n$ for all $T>0$ and $\ep>0$, we have $T_i\to\infty$ for which $\tilde{\mu}_{T_i}$ converges weakly to a probability measure  $\tilde{\mu}^\ep$ as $i\to\infty$. For $\varphi(x,\xi)=L^c(x,\xi)+h(c)$ (in $K$, otherwise re-defined to be $0$ continuously), we have 
\begin{eqnarray*}
\int_{\T^n\times\R^n}L^c(x,\xi)+h(c)d\tilde{\mu}_{T_i}&=&\frac{\ep}{1-e^{-\ep T}}\int^0_{-T_i}e^{\ep s}\{L^c(\gamma^\ep(s),\gamma^\ep{}'(s))+h(c)\}ds\\
&=&\frac{\ep}{1-e^{-\ep T_i}}\{ v^\ep(x_0)-e^{-\ep T_i}v^\ep (\gamma^\ep(-T_i)) \},
\end{eqnarray*}
and hence with $i\to\infty$,
$$\int_{\T^n\times\R^n}L(x,\xi)-c\cdot\xi\,d\tilde{\mu}^\ep=-h(c)+\ep v^\ep(x_0).$$
For $\varphi(x,\xi)=\psi_x(x)\cdot\xi$ (in $K$, otherwise re-defined to be $0$ continuously) with $\psi\in C^1(\T^n)$, we have 
\begin{eqnarray*}
\int_{\T^n\times\R^n}\psi_x(x)\cdot\xi d\tilde{\mu}_{T_i}&=&\frac{\ep}{1-e^{-\ep T}}\int^0_{-T_i}e^{\ep s}\psi_x(\gamma^\ep(s))\cdot\gamma^\ep{}'(s)ds\\
&=&\frac{\ep}{1-e^{-\ep T_i}}\int^0_{-T_i}e^{\ep s} \frac{d}{ds}\{ \psi(\gamma^\ep(s))\}\\
 &=&\frac{\ep}{1-e^{-\ep T_i}}\left\{   e^{\ep s}\psi(\gamma^\ep(s))|^0_{-T_i}-\ep   \int^0_{-T_i}e^{\ep s}  \psi(\gamma^\ep(s)) ds\right\}\\
 &=& \frac{\ep}{1-e^{-\ep T_i}}\left\{  \psi(x_0) - e^{-\ep T_i}\psi(\gamma^\ep(-T_i))\right\}-\ep   \int_{\T^n\times\R^n}  \psi(x) d\tilde{\mu}_{T_i}.
\end{eqnarray*}
Letting $i\to\infty$, we see that $\tilde{\mu}^\ep$ satisfies (\ref{d-holonomic}). In order to see that $\tilde{\mu}^\ep$ is a minimizing measure of (M2)$^\ep$, we check  
\begin{eqnarray}\label{minimini}
\int_{\T^n\times\R^n}L(x,\xi)-c\cdot\xi\,d\mu \ge-h(c)+\ep v^\ep(x_0)
\end{eqnarray}
to be true for any probability measure $\mu$ satisfying (\ref{d-holonomic}). Let $u^\eta$ denote $v^\ep\ast\rho^\eta$, where $\rho^\eta$ is the standard mollifier on $\T^n$ with the parameter $\eta\to0+$. Note that $u^\eta_x=v^\ep_x\ast\rho^\eta$ and  $u^\eta\to v^\ep$ uniformly as $\eta\to0+$. Then, the Legendre transform and Jensen's inequality  imply that  
\begin{eqnarray}\label{LJ}
\qquad \ep u^\eta(x)+\xi\cdot(c+u^\eta_x(x))-L(x,\xi)&\le& \ep u^\eta(x)+H(x,c+u^\eta_x(x))\\\nonumber
&\le& \{\ep v^\ep(\cdot)+H(x,c+v_x^\ep(\cdot))\}\ast\rho^\eta(x)\\\nonumber
&=& \{\ep v^\ep(\cdot)+H(\cdot,c+v_x^\ep(\cdot))\}\ast\rho^\eta(x)+O(\eta)\\\nonumber
&=&h(c)+O(\eta) \mbox{\quad for any $(x,\xi)\in\T^n\times\R^n$}.
\end{eqnarray}
Hence, we have with (\ref{d-holonomic}),
\begin{eqnarray*}
h(c)+O(\eta)&\ge&\ep \int_{\T^n\times\R^n} u^\eta(x)d\mu +\int_{\T^n\times\R^n} \xi\cdot(c+u_x^\eta(x))-L(x,\xi)\,d\mu\\
&=& \ep \int_{\T^n\times\R^n} u^\eta(x)d\mu +\int_{\T^n\times\R^n} \xi\cdot u_x^\eta(x)-\int_{\T^n\times\R^n}L^c(x,\xi)\,d\mu\\
&=&\ep u^\eta(x_0)-\int_{\T^n\times\R^n}L^c(x,\xi)\,d\mu.
\end{eqnarray*}
Letting $\eta\to0+$, we have (\ref{minimini}). 

We examine the support of each minimizing measures of (M2)$^\ep$. Let $\tilde{\mu}^\ep$ be an arbitrary  minimizing measure of (M2)$^\ep$. It follows from (\ref{LJ}) and the minimizing property of $\tilde{\mu}^\ep$ that for each $(x,\xi)\in\supp(\tilde{\mu}^\ep)$ we have
$$\ep u^\eta(x)+\xi\cdot(c+u^\eta_x(x))-L(x,\xi)-h(c)\to0\mbox{ as $\eta\to0+$}.$$
Since $u^\eta_x(\cdot)$ is uniformly bounded in $\T^n$, we have with each accumulating point $p$ of $\{u^\eta_x(x)\}_{\eta>0}$, which is bounded,
$$\ep v^\ep(x)+\xi\cdot(c+p)-L(x,\xi)-h(c)=0\mbox{ \,\,\,\, for each $(x,\xi)\in\supp(\tilde{\mu}^\ep)$}.$$
Hence, $\xi$ is bounded independently of $x$ and $\ep$, namely there exists a compact set $K\subset\T^n\times\R^n$ such that $\supp(\tilde{\mu}^\ep)\subset K$ for any $\ep>0$. If $v^\ep_x(x)$ exists, we have $\xi=H_p(x,c+v^\ep_x(x))$.  
%%%%%%%%%
\begin{Rem}\label{NOT}
A minimizing measure $\tilde{\mu}^\ep$ of {\rm(M2)}$^\ep$  is NOT $\phi^s_{L,c,\ep}$-invariant in general. 
Furthermore, the viscosity solution to \eqref{dHJ} is not differentiable at every point of the support of $\tilde{\mu}^\ep$ in general. 
\end{Rem}
%%%%%%%%%
\noindent 
In order  to check Remark \ref{NOT}, suppose that $v^\ep$ is not differentiable at $x=x_0$.  
Then, a minimizing curve $\gamma^\ep$ for $v^\ep(x_0)$ does not have $x_0$ as its $\alpha$-limit point.  
We will see that the support of the minimizing measure $\tilde{\mu}^\ep$ induced by $\gamma^\ep$ through (\ref{d2-holo11}) contains $(x_0,\xi_0)$ with $\xi_0:=\gamma^\ep{}'(0)$. 
Let $B_\delta$ be a $\delta$-ball of $(x_0,\xi_0)$ and $\varphi$ be a positive function supported inside $B_\delta$. There exists $\tau^-_\delta>0$ such that $\phi^{s}_{L,c,\ep}(x_0,\xi_0)$ first touches $\partial B_\delta$ at $s=\tau^-_\delta$ for $s<0$. Since $x_0$ is not an $\alpha$-limit point of $\gamma^\ep$, $\phi^{s}_{L,c,\ep}(x_0,\xi_0)$ never comes back to $B_\delta$ for $s<\tau^-_\delta$.  Hence, we have
\begin{eqnarray*}
\int_{\T^n\times\R^n}\varphi(x,\xi)d\tilde{\mu}^\ep&=&\lim_{i\to\infty}\frac{\ep}{1-e^{-\ep T_i}}\int^0_{-T_i}e^{\ep s}\varphi(\gamma^\ep(s),\gamma^\ep{}'(s))ds\\
&=&\ep\int^0_{\tau^-_\delta}e^{\ep s}\varphi(\gamma^\ep(s),\gamma^\ep{}'(s))ds>0.
\end{eqnarray*}
Since $\delta>0$ is arbitrary, we have $(x_0,\xi_0)\in\supp(\tilde{\mu}^\ep)$. We  see that $\tilde{\mu}^\ep$ is not $\phi^s_{L,c,\ep}$-invariant. Let $\tau^+_\delta>0$ be such that $\phi^{s}_{L,c,\ep}(x_0,\xi_0)$ first touches $\partial B_\delta$ at $s=\tau^+_\delta$ for $s>0$. Fix $\tau>0$. Take $\delta>0$ so small that $\phi^{-\tau}(B_\delta)\cap B_\delta=\emptyset$. Then $\phi^s_{L,c,\ep}(x_0,\xi_0)$ stays in $\phi^{-\tau}(B_\delta)$ only within $s\in(-\tau-\tau^-_\delta,-\tau+\tau^+_\delta)$ for $\delta\to0+$. Hence, with the indicator function $\chi_A$ for each $A\subset\T^n\times\R^n$ and a standard mollifier $\rho^\eta$ on $\T^n\times\R^n$, we have 
\begin{eqnarray*}
\tilde{\mu}^\ep(B_\delta)&=&\lim_{i\to\infty}\lim_{\eta\to0+}\int_{\T^n\times\R^n}\left(\chi_{B_\delta}\ast\rho^\eta\right)(x,\xi)d\tilde{\mu}_{T_i}\\
&=&\lim_{i\to\infty}\lim_{\eta\to0+}\frac{\ep}{1-e^{-\ep T_i}}\int^0_{-T_i}e^{\ep s}\left(\chi_{B_\delta}\ast\rho^\eta\right)(\gamma^\ast(s),\gamma^\ast{}'(s))ds\\
&=&1-e^{-\ep\tau^-_\delta},\\
\tilde{\mu}^\ep(\phi^{-\tau}_{L,c,\ep}(B_\delta))&=&\lim_{i\to\infty}\lim_{\eta\to0+}\int_{\T^n\times\R^n}\left(\chi_{\phi^{-\tau}_{L,c,\ep}(B_\delta)}\ast\rho^\eta\right)(x,\xi)d\tilde{\mu}_{T_i}\\
&=&\lim_{i\to\infty}\lim_{\eta\to0+}\frac{\ep}{1-e^{-\ep T_i}}\int^0_{-T_i}\left(e^{\ep s}\chi_{\phi^{-\tau}_{L,c,\ep}(B_\delta)}\ast\rho^\eta\right)(\gamma^\ast(s),\gamma^\ast{}'(s))ds\\
&=&e^{-\ep(\tau-\tau^+_\delta)}-e^{-\ep(\tau+\tau^-_\delta)},
\end{eqnarray*}
which are not identical. This implies also that $v^\ep$ is not necessarily differentiable at every point of $\pi\,\supp(\tilde{\mu}^\ep)$, e.g., at $x_0$ in the above example.  

Define the sets for each $\ep>0$,
$$\tilde{\hat{\mathcal{M}}}^\ep(c ):= \overline{\bigcup_{\tilde{\mu}^\ep} \supp(\tilde{\mu}^\ep)},\,\,\,\,\hat{\mathcal{M}}^\ep(c ):= \pi \tilde{\hat{\mathcal{M}}}^\ep(c ),$$
where the union is taken with respect to all minimizing measures of (M2)$^\ep$ for all $x_0\in\T^n$. 
Unlike (M1)$^\ep$, minimizers of (M2)$^\ep$ do not provide information on $\alpha$-limit points, and in particular,
$$\mbox{$\tilde{\hat{\mathcal{M}}}^\ep(c )\not\subset\tilde{\mathcal{M}}_\alpha^\ep(c)$, \,\,\,${\hat{\mathcal{M}}}^\ep(c )\not\subset{\mathcal{M}}_\alpha^\ep(c)$ in general.}$$
Nevertheless, since each minimizing measure  of (M2)$^\ep$ is supported on a compact set $K$ for any $\ep>0$, there exists weakly convergent sequence $\tilde{\mu}^{\ep_i}$ with $\ep_i\to0+$ ($i\to\infty$). Since (\ref{d-holonomic}) yields the holonomic condition (\ref{holonomic}) for $\ep\to0+$, the limit of $\tilde{\mu}^{\ep_i}$ is a minimizing measure of (M).           
%%%%%%%%%%%%%
%%%%%%%%%%%%%
%%%%%%%%%%%%%
%%%%%%%%%%%%%%%%
\setcounter{section}{3}
\setcounter{equation}{0}
\section{Error estimate in vanishing discount process}
We show error estimates between viscosity solutions of (\ref{HJ}) and (\ref{dHJ})  in the following two cases: 
\begin{enumerate}
 \item[(C1)] (\ref{HJ}) with $H(x,p)=\frac{1}{2}p^2-F(x):\T\times\R\to\R$  whose corresponding Hamiltonian system possesses hyperbolic stationary solutions. 
 \item[(C2)] (\ref{HJ}) with a $C^2$-solution which generates a KAM $n$-torus. 
\end{enumerate} 
In the first case, we could see the advantage to introducing the families of $\alpha$-limit points $\mathcal{M}_\alpha(c)$ and $\mathcal{M}_\alpha^\ep(c)$, where asymptotics of the discount dynamical systems including $\mathcal{M}_\alpha^\ep(c)$ for $\ep\to0+$ is also made clear.   

We refer to \cite[Theorem 1.1]{Davini}, \cite[Theorem 1.1]{MT} for the general convergence result in the vanishing discount process.   
Under (H1)-(H3), there exists a viscosity solution $v^\ast$ to \eqref{HJ} such that 
\[
v^\ep\to v^\ast\mbox{ uniformly on $\T^n$ as $\ep\to0+$}. 
\]
Moreover, $v^\ast$ is characterized by 
\begin{multline}\label{rep-form-asym}
v^\ast(x)=\sup\{v(x)\,|\,\mbox{$v\in {\rm Lip\,}(\T^n)$, viscosity  subsolution to (\ref{HJ}) such that}\\
\qquad\qquad \ \int_{\T^n\times\R^n}v\,d\mu\le0 \ \text{for all Mather measures}  \mu \ \text{of (M)}\}. 
\end{multline}

We remark that a criterion in this characterization is first observed in \cite[Corollary 4]{G}. We also refer to \cite{CCG} for an $L^2$-estimate with respect to a discount Mather measure between the gradient of the viscosity solution to \eqref{dHJ} and that of the mollified function of a viscosity solution to \eqref{HJ} on the support of the measure.

%%%%%%%%%%%%%%%%%%%%%%
\subsection{Error estimate in (C1)}
  We examine asymptotics of the discount problem for $\ep\to0+$ with    
\begin{eqnarray*}
H(x,p):=\frac{1}{2}p^2-F(x):\T\times\R\to\R,\quad F\in C^2(\T), \quad F(x)\ge 0\mbox{\quad in $\T$}.
\end{eqnarray*}  
Suppose that $F$ has a finite  number of zero points $0\le x_1<x_2<\ldots<x_I<1$ at which $F_{xx}(x_i)>0$ for all $i=1,2,\ldots,I$. Since $F_x(x_i)=0$, the points $(x_i,0)$ are hyperbolic stationary solutions of the Hamiltonian system (\ref{HS}), whose stable/unstable manifolds form a separatrix  
$$S_\pm:=\{ (x,\pm\sqrt{2F(x)})\,|\,x\in\T\}.$$   
Note that the slope of $S_\pm$ at $x_i$ is $\pm\sqrt{F_{xx}(x_i)}$. Set $c_\pm:=\pm\int_\T\sqrt{2F(x)}dx$. Due to such a structure of the Hamiltonian system, we see that the Hamilton-Jacobi equation  (\ref{HJ}) has the following properties: 
\begin{itemize}
\item $h(c )\equiv0$ for $c\in[c_-,c_+]$.
\item $\mathcal{M}_\alpha(v;c)=\{x_1,x_2,\ldots,x_I\}$ for each $c\in[c_-,c_+]$ and viscosity solution $v$ of (\ref{HJ}).
\item Let $c\in(c_-,c_+)$ and $v$ be a viscosity solution of (\ref{HJ}). Then the graph of $c+v_x$ consists of parts of $S_\pm$ and has at least one point of discontinuity in such a way that the jump is only from $S_+$ to $S_-$ (the entropy condition).   
\item Let $c\in(c_-,c_+)$. If a function $u(x)$ is such that $\int_\T c+u(x)dx=c$ and the graph of $c+u(x)$ consists of parts of $S_\pm$, and has discontinuity only from $S_+$ to $S_-$, then the primitive function of $u$ is a viscosity solution of (\ref{HJ}).
\item If $I\ge2$, such functions as $u$ may have $2$ to $I$ points of singularity and may exist uncountably many (shift the position of discontinuity right or left keeping $\int_\T c+u(x)dx=c$).   This means that there are uncountably many viscosity solutions of (\ref{HJ}) beyond constant difference.    
\end{itemize}
Note carefully that the points between $x_{i-1}$ and $x_i$ at which $F_x=0$ also give stationary  solutions of (\ref{HS}), where they are elliptic if $F_{xx}<0$ or hyperbolic if $F_{xx}>0$.

\indent We study the limit process $\ep\to0+$ of (\ref{dHJ}) for $c\in (c_-,c_+)$, exploiting dynamical properties of (\ref{dHS}) in the present situation. The first task is to specify $\mathcal{M}_\alpha^\ep(c)$. Here, $(\ref{dHS})$ is of the form
\begin{eqnarray}\label{1dHS}
x'(s)=p(s),\quad p'(s)=F_x(x(s))+c\ep-\ep p(s). 
\end{eqnarray} 
For each small $\ep>0$, (\ref{1dHS}) has  hyperbolic stationary solutions $(x_i^{c,\ep},0)$ which are $O(\ep)$-close to $(x_i,0)$ for $i=1,2,\ldots,I$; namely, $x_i^{c,\ep}$ is the value of  the implicit function derived from $F_x(x)+c\ep=0$ near $(\ep,x)=(0,x_i)$. Therefore we have the local stable/unstable manifolds of $(x_i^\ep,0)$, which tend to the exact local stable/unstable manifolds of $(x_i,0)$ as $\ep\to0$. Note carefully that the geometrical structure of the phase space of (\ref{1dHS}) is not equivalent to that of (\ref{HJ}) no matter how small $\ep>0$ is, i.e., the stable/unstable manifolds of $(x_i^\ep,0)$ (the extension of the local ones by $\phi^s_{H,c,\ep}$ for $s\in\R$) do not connect each other and hence do not form a separatrix like $S_\pm$ in general, where some of the stationary solutions of (\ref{HS}) may change to be asymptotically stable due to non-Hamiltonian perturbation in (\ref{1dHS}).  However the stable/unstable manifolds cannot be transversal in our autonomous $1$-dimensional setting. 
%%%%%%%%%%%
\begin{Thm}\label{d-Mather1}
 Set $\Gamma^{c,\ep}:=\{x_1^{c,\ep},\ldots,x_I^{c,\ep}\}$, where $x_i^{c,\ep}$ is the value of the implicit function derived from $F_x(x)+c\ep=0$ near $(\ep,x)=(0,x_i)$. For each sufficiently small $\ep>0$, it holds that 
\begin{enumerate}[{\rm(i)}]
\item $c+v_x^\ep(x)=0$ on $\mathcal{M}_\alpha^\ep(c)$ for each $c\in (c_-,c_+)$.
\item $\mathcal{M}_\alpha^\ep(c)\subset \Gamma^{c,\ep}$ for each $c\in (c_-,c_+)$.  
\item $|v^\ep|\le \beta \ep $ on $\mathcal{M}_\alpha^{\ep}(c)$  for each $c\in (c_-,c_+)$, where $\beta>0$ is a constant independent of $c$ and $\ep$.
\item There exists $c^\ast_+\in (0,c_+]$ and $c^\ast_-\in[c_-,0)$ such that $\mathcal{M}_\alpha^\ep(c)= \Gamma^{c,\ep}$ for all $c\in[c^\ast_-,c^\ast_+]$. 
%Furthermore, for the viscosity solution $v^\ep$ of (\ref{dHJ}) with $c=0$, $\graph(v_x^\ep)$ has a discontinuity within each interval $[x^{0,\ep}_i,x^{0,\ep}_{i+1}]\mod1$ switching from the unstable manifold of  $(x^{0,\ep}_i,0)$ to the unstable manifold of $(x^{0,\ep}_{i+1},0)$ 
\end{enumerate}  
\end{Thm} 
%%%%%%%%%%%
\begin{proof}
(i) Fix $c\in(c_-,c_+)$. Suppose that there exists a sequence ${\ep_i}\to0+$ ($i\to\infty$) for which  $\mathcal{M}_\alpha^\ep(c)$ contains a point $x^\ast$ such that $c+v_x^{\ep_i}(x^\ast)\neq0$ ($v^{\ep_i}$ is differentiable on $\mathcal{M}_\alpha^{\ep_i}(c)$, as  shown in Section 3). Let $\gamma^{\ep_i}(s):(-\infty,0]\to \T$ be a minimizing curve that has $x^\ast$ as its $\alpha$-limit point. Then $(x(s),p(s)):=(\gamma^{\ep_i} (s),c+v_x^{\ep_i}((\gamma^{\ep_i} (s)))$, which solves (\ref{1dHS}) for $s\le0$, is not on any unstable manifolds of (\ref{1dHS}) nor in the region that is absorbed by asymptotically unstable stationary solutions of (\ref{1dHS}), because otherwise we necessarily have $\lim_{s\to-\infty}(x(s),p(s))=(x^\ast,0)=(x^\ast,c+v_x^{\ep_i}(x^\ast))$.   
Moreover $(x(s),p(s))$ cannot touch the line $p=0$ for $s<0$.  
If not, $(x(s),p(s))$ touches the line $p=0$ away from any stationary solution of (\ref{1dHS}) at some $s=s_0$.  
Then $x''(s_0)=p'(s_0)=F_x(x(s_0))+c{\ep_i}\neq0$, $x'(s_0)=0$ and $x(s)$ takes a local maximum or minimum at $s=s_0$. 
Hence $p(s)$ changes its sign at $s=s_0$ and $\{ (x(s),p(s))\,\,|\,\,s\in[s_0-\delta,s_0+\delta]\}$ cannot be a single-valued graph defined on $\{x(s)\,\,|\,\,s\in[s_0-\delta,s_0+\delta]\}$. 
However $(x(s),p(s))$ must be on $\graph(c+v_x^{\ep_i})$ for $s<0$, and we reach a contradiction.  Hence,  $\{ (x(s),p(s))\,\,|\,\,s<0\}\subset \graph(c+v_x^{\ep_i})$ must be strictly above or below the line $p=0$, where $x(s)$ moves over $\T$ for $s<0$. 

Therefore, we conclude that  $\graph(c+v_x^{\ep_i})$ is smooth and strictly above or below the line $p=0$ for all $i$. Then, we have a subsequence of $c+v_x^{\ep_i}$ that converges to $c+v_x$ pointwise a.e. as $i\to\infty$ (we do not yet assert that the whole sequence is convergent), where $v$ is a viscosity solution. $\graph(c+v_x)$ is above or below the line $p=0$, which is only possible for $c=c_\pm$.  Thus we obtain the conclusion.     

(ii)  Fix $c\in(c_-,c_+)$. Let $\ep_0>0$ and $\delta>0$ be such that the implicit function $x=x_i^{c,\ep}:[-\ep_0,\ep_0]\to[x_i-\delta,x_i+\delta]$ is one to one and onto for all $i=1,\ldots,I$.  Since $F(x)=0$ only for $x=x_1,\ldots,x_I$, there exists $\eta>0$ such that $F(x)>\eta$ for all $x\not\in [x_i-\delta,x_i+\delta]$, $i=1,\ldots,I$. 
Let  $x^\ast$ be an arbitrary point of $\mathcal{M}_\alpha^\ep(c)$. Since $c+v_x^\ep(x^\ast)=0$ due to (i) and $v^\ep$ is uniformly bounded, it follows from (\ref{dHJ}) that for sufficiently small $\ep>0$,
$$F(x^\ast)=\ep v^\ep(x^\ast)\le\eta.$$
Hence, $x^\ast\in[x_i-\delta,x_i+\delta]$ for some $i$. Since $(x^\ast,0)$ is a stationary solution  of (\ref{1dHS}), we have $F_x(x^\ast)+c\ep=0$, which implies $x^\ast=x_i^{c,\ep}$. 

(iii)  In what follows, $\beta_j>0$ are some constants independent of $\ep$ and $c$. Let  $x^\ast$ be an arbitrary point of $\mathcal{M}_\alpha^\ep(c)$.  Since $x^\ast$ is equal to $x_i^{c,\ep}$ for some $i$ due to (ii), we have $|x^\ast-x_i|\le \beta_1\ep$. Hence, we observe with (i) and (\ref{dHJ}) that 
\begin{eqnarray*}
\ep v^\ep(x^\ast)=F(x^\ast)=F(x^\ast)-F(x_i)=F_x(x_i)+\frac{1}{2}F_{xx}(x_i+\theta (x^\ast-x_i))(x^\ast-x_i)^2,
\end{eqnarray*}
where $\theta\in(0,1)$. Since $F_x(x_i)=0$, we conclude that $|v^\ep(x^\ast)|\le \beta_2 \ep$.

(iv) Suppose that $\Gamma^{c,\ep}\setminus \mathcal{M}_\alpha^\ep(c)$ is non-empty. Let  $x^\ast$ be an arbitrary point of $\Gamma^{c,\ep}\setminus \mathcal{M}_\alpha^\ep(c)$. Since $x^\ast$ is equal to $x_i^{c,\ep}$ for some $i$, we have $|x^\ast-x_i|\le \beta_1\ep$. By (\ref{dHJ}), we observe that if $v^\ep$ is differentiable at $x^\ast$ 
\begin{eqnarray*}
\ep v^\ep(x^\ast)&=&F(x^\ast)-\frac{1}{2}(c+v_x^\ep(x^\ast))^2\\
&\le& F(x^\ast)=F_x(x_i)+\frac{1}{2}F_{xx}(x_i+\theta (x^\ast-x_i))(x^\ast-x_i)^2
\le \beta_2  \ep^2
\end{eqnarray*}
for some $\theta\in(0,1)$, 
and hence we have 
\begin{eqnarray}\label{4112}
v^\ep(x^\ast)\le \beta_2\ep.
\end{eqnarray}
If $v^\ep$ is not differentiable at $x^\ast$, take a point of differentiability of $v^\ep$ near $x^\ast$ and use continuity of $v^\ep$ to obtain (\ref{4112}).   

Since $x^\ast\not\in\mathcal{M}_\alpha^\ep(c)$, the minimizing curve $\gamma^\ep:(-\infty,0]\to\R$ (we do not take ``mod $1$'' now) starting from a point $x$ of differentiability of $v^\ep$ near $x^\ast$ is such that $(\gamma^\ep(s),c+v_x^\ep(\gamma^\ep(s)))$ stays on the unstable manifold of a certain stationary solution $(x_j^{c,\ep},0)$ of (\ref{1dHS}) for $s\le0$, where $x_j^{c,\ep}\in \mathcal{M}_\alpha^\ep(c )$. Due to the argument in the proof of (i),  $(\gamma^\ep(s),c+v_x^\ep(\gamma^\ep(s)))$ cannot touch the line $p=0$ for $s\le0$. If $c+v_x^\ep(x)>0$ (resp., $c+v_x^\ep(x)<0$),  then $(\gamma^\ep(s),c+v_x^\ep(\gamma^\ep(s)))$ is above (resp.,  below) the line $p=0$ and falls into  $(x_j^{c,\ep}+m, 0)$ with $m=0$ or $-1$ from the right (resp.,  $m=0$ or $1$ from the left) as $s\to-\infty$. Therefore, sending $x$ to $x^\ast$, we see  that 
\begin{eqnarray*}
\beta_3&<& \int _{x_j^{c,\ep}+m}^{x^\ast} c+ v^\ep_x(x) dx=c(x^\ast-x_j^{c,\ep}-m)+v(x^\ast)-v^\ep(x_j^{c,\ep})\\
(\mbox{resp., } -\beta_3&>& \int ^{x_j^{c,\ep}+m}_{x^\ast}c+ v^\ep_x(x) dx=c(x_j^{c,\ep}+m-x^\ast)+v^\ep(x_j^{c,\ep})-v^\ep(x^\ast)),
\end{eqnarray*}
where $\beta_3>0$ is a number smaller than the area trimmed by the line $x=x^\ast$ and the unstable manifold of $(x_j^{0,\ep}+m, 0)$ within $x\in[x_j^{0,\ep}+m, x^\ast]$ (resp.,   within $x\in[x^\ast,x_j^{0,\ep}+m]$).  Here $\beta_3$ is of $O(1)$ for $\ep\to0$. Since $|v^\ep(x_j^{c,\ep})|\le \beta \ep$ due to (iii), we have 
\begin{eqnarray*}
&&\beta_2\ep\ge v^\ep(x^\ast)>\beta_3-\beta \ep - c(x^\ast-x_j^{c,\ep}-m)\\
&&(\mbox{resp., } \beta_2\ep\ge v^\ep(x^\ast)>\beta_3-\beta \ep + c(x_j^{c,\ep}+m-x^\ast) ).
\end{eqnarray*}
If $|c|$ is less than a certain value,  we reach a contradiction. 
%We conclude the proof by showing that the set of all points $c\in[c_-,c_+]$ for which $\mathcal{M}_\alpha^\ep(c_j)= \Gamma^{c_j,\ep}$ holds, is closed. It follows from (\ref{ddd-value}) that the viscosity solution of (\ref{dHJ}), denoted here by $v^{c,\ep}$, is continuous with respect to $c$.  
%Fix a sufficiently small $\ep>0$. If $|c|$ is small, a similar reasoning to the proof of 4. of Lemma \ref{d-Mather1} is available and $\mathcal{M}_\alpha^\ep(c)= \Gamma^{c,\ep}$. Let $c^\ast$ be the supremum of all $c\in[0,c_+]$ such that $\mathcal{M}_\alpha^\ep(c)= \Gamma^{c,\ep}$ holds. Take a sequence $c_j\nearrow c^\ast$ ($j\to\infty$) such that $\mathcal{M}_\alpha^\ep(c_j)= \Gamma^{c_j,\ep}$ for all $j$.  By Lemma \ref{d-Mather1}, we see that $c_j+v^{c_j,\ep}_x(x^{c_j,\ep}_i)=0$ for all $i=1,\ldots,I$. By (\ref{dHJ}) and  continuity  with respect to $c$, we have 
%$$(c^\ast+v_x^{c^\ast,\ep}(x^{c^\ast,\ep}_i))^2 -(c_j+v_x^{c_j,\ep}(x^{c_j,\ep}_i))^2=-\ep (v^{c^\ast,\ep}(x^{c^\ast,\ep})-v^{c_j,\ep}(x^{c_j,\ep}))\to0\mbox{ as $j\to\infty$}.$$  
%Therefore $c^\ast+v_x^{c^\ast,\ep}(x^{c^\ast,\ep}_i)=0$ for all $i=1,\ldots,I$ and $\mathcal{M}_\alpha^\ep(c^\ast)= \Gamma^{c^\ast,\ep}$.
\end{proof}
%%%%%%%%%%%%%%
\begin{Rem}\label{same} 
$\mathcal{M}_\alpha^\ep(c)= \Gamma^{c,\ep}$ does not hold for every $c\in(c_-,c_+)$ in general, 
which implies that not every Mather measure is obtained through  {\rm(M1)}$^\ep$ in the vanishing discount process.
\end{Rem}
%%%%%%%%%%
\noindent For more detailed discussion in regards to Remark \ref{same}, let us consider the following example: $F\in C^2(\T)$, $F(x)\ge0$, $F(x)=0$ for $x\in\{0,\,\,1/2\}$, $c_\pm:=\pm\int_{\T}\sqrt{2F(x)}dx=\pm2$, $S_1:=\int_0^{1/2}\sqrt{2F(x)}dx=1/2$. Set $x_1=0$ and  $x_2=1/2$. Then, for each sufficiently small $\ep>0$, we have
$$\mathcal{M}_\alpha^\ep(c)=\{x_1^{c,\ep}\}\,\,\,\,\mbox{for }c\in(1,2),\quad \mathcal{M}_\alpha^\ep(c)=\{x_2^{c,\ep}\}\,\,\,\,\mbox{for }c\in(-2,-1).$$
In fact, let $c\in(1,2)$ (resp., $c\in(-2,-1)$) and suppose that there exists $\ep_j\to 0$ such that   
$x_2^{c,\ep_j}\in \mathcal{M}_\alpha^{\ep_j}(c)$ (resp., $x_1^{c,\ep_j}\in \mathcal{M}_\alpha^{\ep_j}(c)$ ) for all $j$. By (\ref{d-value-func}) with $\gamma(s)\equiv x_1,x_2$ and letting $T\to\infty$, we have $v^{\ep_j}(x_1)\le0$, $v^{\ep_j}(x_2)\le0$. Hence, an accumulating point $v^\ast$ of $\{v^{\ep_j}\}$ satisfies $v^{\ast}(x_1)\le0$, $v^{\ast}(x_2)\le0$. In particular, by (iii) of Theorem \ref{d-Mather1}, we have $O(\ep_j)=v^{\ep_j}(x_2^{c,\ep_j})=v^{\ep_j}(x_2)+O(\ep_j)$ (resp., $O(\ep_j)=v^{\ep_j}(x_1^{c,\ep_j})=v^{\ep_j}(x_1)+O(\ep_j)$), and hence $v^\ast(x_2)=0$ (resp., $v^\ast(x_1)=0$). Since  $\graph(c+v^\ast_x)$ consists of parts of the graph of $\pm\sqrt{2F(x)}$ with  jump down discontinuity, we see that  
\begin{eqnarray*}
\frac{1}{2}=S_1&\ge&|\int_{x_1}^{x_2}c+v^\ast_x(x)\,dx|=|c(x_2-x_1)+v^\ast(x_2)-v^\ast(x_1)|=|\frac{c}{2}-v^\ast(x_1)|\ge\frac{c}{2},\\
\mbox{(resp., }\frac{1}{2}&\ge&|\int_{x_1}^{x_2}c+v^\ast_x(x)\,dx|=|c(x_2-x_1)+v^\ast(x_2)-v^\ast(x_1)|=|\frac{c}{2}+v^\ast(x_2)|\ge
-\frac{c}{2}).
\end{eqnarray*}
If $c>1$ (resp., $c<-1$), we reach a contradiction. 

In this example, Mather measures are Dirac measures supported by $x_1$ and $x_2$. It is worth emphasizing that we can clearly see the selection of Mather measures by (M1)$^\ep$ depending on $c$.

Notice here that the selection of viscosity solutions and Mather measures in the vanishing discount process is totally different from that in the vanishing viscosity process. 
Indeed, if we assume $0<F_{xx}(x_1)< F_{xx}(x_2)$, then for each $c\in(-2,2)$ the vanishing viscosity method selects the viscosity solution $\tilde{v}$ up to constant such that $\graph(c+\tilde{v}_x)$ contains the local unstable manifold of $(x_i,0)$ in a neighborhood of $x_i$ only for $i=1$, which has only one point of discontinuity \cite{JKM}, \cite{Bessi}, \cite{A} (cf.,  if $c=0$, for instance, $\graph(c+v^\ast_x)$ has two points of discontinuity due to $v^\ast(x_1)=v^\ast(x_2)=0$, which 
can also be observed by the formula \eqref{rep-form-asym}). 
We can also see a difference of the selection of Mather measures associated with the vanishing viscosity process (see \cite{A1} for details). 
%If $F_{xx}(x_1)= F_{xx}(x_2)$, we do not yet know which viscosity solution is selected by the vanishing viscosity method. If $F_{xx}(x_1)\neq F_{xx}(x_2)$ (say, $F_{xx}(x_1)< F_{xx}(x_2)$), then for each $c\in(-2,2)$ the vanishing viscosity method selects the viscosity solution $v^{\ast\ast}$ up to constant such that $\graph(c+v_x^{\ast\ast})$ contains the local unstable manifold of $(x_i,0)$ in a neighborhood of $x_i$ only for $i=1$, which has only one point of discontinuity \cite{JKM}, \cite{Bessi}, \cite{A} (cf.,  if $c=0$, for instance, $\graph(c+v^\ast_x)$ has two points of discontinuity due to $v^\ast(x_1)=v^\ast(x_2)=0$).     
%%%%%%%%%%%
%\noindent In order to check Remark \ref{same}, Consider the following example of $F(x)$: $F(x)\ge0$, $F(x)=0$ for $x\in\{0,\,\,1/2\}$, $c_\pm:=\pm\int_{\T}\sqrt{2F(x)}dx=\pm2$, $S_1:=\int_0^{1/2}\sqrt{2F(x)}dx=1/2$. Here $x_1=0$ and  $x_2=1/2$. 
%Suppose that $\mathcal{M}_\alpha^\ep(c)= \Gamma^{c,\ep}$ holds for every $c\in(c_-,c_+)$. Then, by (iii) of Theorem \ref{d-Mather1}, $v^\ep(x)=O(\ep)$ on $\mathcal{M}_\alpha^\ep(c)$. Since $v^\ep$ is uniformly Lipschitz, $v^\ep(x)=O(\ep)$ on  $\mathcal{M}_\alpha(c)=\{x_1,x_2\}$.   Hence, an accumulating point $v$ of $\{v^\ep\}_{\ep>0}$ satisfies $v(x)=0$ on $\mathcal{M}_\alpha(c)$, where $v$ is a viscosity solution of (\ref{HJ}). Note that $\graph(c+v_x)$ consists of parts of the graph of $\pm\sqrt{2F(x)}$ with  jump down discontinuity.   Then, 
%$$S_1\ge|\int_{x_1}^{x_2}c+v_x(x)dx|=|c(x_2-x_1)|=\frac{|c|}{2}.$$
%Hence we have $|c|\le1$. If $|c|>1$, we reach a contradiction.

%%%%%%%%%%
Now it is easy to observe that 
\begin{itemize}
\item  For $c\in[c^\ast_-,c^\ast_+]$, $v^\ep(x)=O(\ep)$ on $\{x_1,\ldots,x_I\}$ for any $\ep>0$ and any accumulating point $v^\ast$ of $\{v^\ep\}_{\ep>0}$ vanishes on $\mathcal{M}_\alpha(c)=\{x_1,\ldots,x_I\}$. Such $v^\ast$ is unique because of (iv) of Theorem \ref{invariance} and therefore $v^\ep\to v^\ast$ as $\ep\to0$ in the whole sequence. This is a specific example of the general convergence result in \cite{Davini} and \cite{MT}.
\item There may exist $c\in(c_-,c_+)$ for which an accumulating point $v^\ast$ of $\{v^\ep\}_{\ep>0}$ does not vanishes on $\{x_1,\ldots,x_I\}$.  
\end{itemize}

Next, we obtain error estimates between $v^\ep$ and $v^\ast=\lim_{\ep\to0+}v^\ep$. Our strategy is to estimate the time for each minimizing curve to fall into $\mathcal{M}_\alpha^\ep(c), \mathcal{M}_\alpha(c)$ along the unstable manifolds and to compare the value functions (\ref{value-func}) and (\ref{d-value-func}). In what follows,  $x_{I+k}:=x_{k}+1$ and each interval $K\subset\R$ is identified with $K\mod1:=\{x\mod1\,|\,x\in K\}$. 

Let $\bar{x}_i$ be the midpoint of the interval $[x_i,x_{i+1}]$, $i=1,2,\ldots,I$.  Since the slope of the local stable/unstable manifold of each $(x_i,0)$, given by the graph of $\pm\sqrt{2F(x)}$, is strictly away from $0$ at $(x_i,0)$, there exists $b>0$ such that 
$$\mbox{$\sqrt{2F(x)}\ge b(x-x_i)$ on $[x_i,\bar{x}_i]$, \,\, $\sqrt{2F(x)}\ge -b(x-x_{i+1})$ on $[\bar{x}_i,x_{i+1}]$}.$$
Hence, each trajectory $(x(s),p(s))$ on the unstable manifold of $(x_i,0)$ with $x(0)\in(x_i,x_{i+1})$ satisfies the estimate 
\begin{eqnarray*}
&&\mbox{if $x(0)\in(\bar{x}_i,x_{i+1}]$, }\\
&&\qquad x'(t)=p(t)\ge -b(x(t)-x_{i+1}) \mbox{ and } x_{i+1}-x(-\tau)\ge(x_{i+1}-x(0))e^{b\tau},\\
&&\mbox{if $x(0)\in[x_i,\bar{x}_{i})$, }\\
&&\qquad x'(t)=p(t)\ge b(x(t)-x_{i}) \mbox{ and } 0<x(-\tau)-x_{i}\le(x(0)-x_{i})e^{-b\tau}.
\end{eqnarray*} 
Therefore, for each $\delta>0$, the time $\tau>0$ for which the backward trajectory with $x(0)\in(\bar{x}_i,x_{i+1}-\delta]$ reaches the midpoint $(\bar{x}_i,\sqrt{2F(\bar{x}_i)})$ is estimated as
\begin{eqnarray}\label{estimate} 
\tau\le\frac{1}{b}\log\frac{1}{\delta}.
\end{eqnarray}   
The same estimate holds for the time $\tau>0$ for which the backward trajectory with $x(0)\in(\bar{x}_i+\delta,\bar{x}_{i}]$ reaches the $\delta$-cylinder $[x_i-\delta,x_i+\delta]\times \R$. In this way, we see that each trajectory on the stable/unstable manifolds of $(x_i,0)$ starting away from  $\delta$-cylinders $[x_i-\delta,x_i+\delta]\times \R$ ($i=1,\ldots,I$) reaches a $\delta$-cylinder $[x_j-\delta,x_j+\delta]\times \R$ ($j=i-1$, $i$ or $i+1$) taking at most  time $T>0$:    
\begin{eqnarray}\label{estimate2} 
T\le\frac{2}{b}\log\frac{1}{\delta}.
\end{eqnarray}   

Consider the case with $\ep>0$. For each $i=1,\ldots,I$, the stable/unstable manifold of $(x^{c,\ep}_i,0)$ grows from $(x_i^{c,\ep},0)$ and has the first contact with either the line $p=0$ at $(x^\ast,0)$, where $x^\ast$ is not necessarily certain $(x_j^{c,\ep},0)$, or the line $x=x_i^{c,\ep}$ at $(x^{c,\ep}_i,p^\ast)$ after passing over whole $\T$, where $p^\ast$ is not necessarily $0$. Let $(x(s),p(s))$ be any trajectory on the part of the stable/unstable manifold connecting $(x_i^{c,\ep},0)$ and $(x^\ast,0)$ or connecting $(x_i^{c,\ep},0)$ and $(x_i^{c,\ep},p^\ast)$. Consider $\delta$-cylinders $[x_j^{c,\ep}-\delta,x_j^{c,\ep}+\delta]\times \R$, $j=1,\ldots,I$ and the $\delta$-cylinder $[x^\ast-\delta,x^\ast+\delta]\times \R$.  Since the local stable/unstable manifolds of $(x_i^{c,\ep},0)$ are close to those of $(x_i,0)$ for sufficiently small $\ep>0$,  we see, with a smaller $b>0$ if necessary, that the above trajectory $(x(s),x(s))$ takes at most time $T>0$ estimated as (\ref{estimate2}) to first touch one of the above cylinders after leaving a point outside of any of the cylinders.  
%%%%%%%
\begin{Thm}\label{error22}
Let $\delta(\ep)>0$ be any function tending to $0+$ with $\delta(\ep)\ge\alpha\ep$ as $\ep\to0+$, where $\alpha$ is a constant. Then, for each $c\in(c_-,c_+)$, the limit $v^\ast=\lim_{\ep\to0+}v^\ep$ satisfies 
$$|v^\ep(x)-v^\ast(x)|\le \beta(\ep |\log\delta(\ep)|^2+\delta(\ep))  \mbox{  \,\,\,on $\T$ \,\,\,\,as $\ep\to0+$,}$$
where $\beta>0$ is a constant independent of $c,\ep$. 
\end{Thm}
%%%%%%%
\begin{Rem} 
For each $\nu\in(0,1)$, we may take $\delta(\ep)=\ep^\nu$ obtaining $|v^\ep(x)-v^\ast(x)|\le \tilde{\beta}\ep^\nu$ as $\ep\to0+$.
\end{Rem}
%%%%%%%%%
\begin{proof}[{\it Proof of Theorem {\rm\ref{error22}}}]
We compare the two value functions (\ref{value-func}) and (\ref{d-value-func}), where $h(c )=0$ and $L(x,\xi)=\frac{1}{2}\xi^2+F(x)$. In what follows, $\alpha_i>0$ are some constants independent of $\ep$ and $c$. For any $\ep\in(0,\ep_0)$ with a sufficiently small $\ep_0>0$, it holds that $|x^{c,\ep}_i-x_i|< \alpha_1\ep$ for all $i$ and $c$. Take a number $\delta\ge\alpha_1\ep$.

First we deal with the case of $c\in[c_-^\ast,c_+^\ast]$, where $\mathcal{M}_\alpha^\ep(c)=\Gamma^{c,\ep}$ and $|v^\ep(x)|\le \alpha_2 \ep$ on $\mathcal{M}_\alpha^\ep(c )$ for $\ep\to0+$. Hence, $v(x)=0$ on $\mathcal{M}_\alpha(c)=\{x_1,\ldots,x_I\}$.  Since $v$ and $v^\ep$ are Lipschitz, we see that 
\begin{equation}\label{41}
|v^\ep(x)-v^\ast(x)|\le \alpha_2 \delta \quad\mbox{ on $[x_i-2\delta,x_i+2\delta]$ \quad for $i=1,\ldots,I$.}
\end{equation}
Suppose that the unstable manifold growing from $(x_i^{c,\ep},0)$ first touches the line $p=0$  away from any $(x_j^{c,\ep},0)$, where the contact point (there are two such contact points at most for each $i$) is denoted by $(x^\ast,0)$,  and contains a point of $\graph(c+v^\ep_x)$ in the $\delta$-cylinder $(x^\ast-\delta,x^\ast+\delta)\times\R$. Then, we define the interval $K^{c,\ep}_i(x^\ast)$ as $K^{c,\ep}_i(x^\ast):=(x^\ast-\delta,x^\ast+\delta)$. Let $K^{c,\ep}$ denote the union of all such interval $K^{c,\ep}_i(x^\ast)$ over $i=1,\ldots,I$, where if there is no such $K^{c,\ep}_i(x^\ast)$, we define $K^{c,\ep}:=\emptyset$. Let $x$ be an arbitrary point of $\T\setminus (\cup_{1\le i\le I} [x_i-2\delta,x_i+2\delta]\cup K^{c,\ep})$ and  let $\gamma^\ast$, $\gamma^\ep$ be a minimizing curve for $v(x)$, $v^\ep(x)$, respectively. Then $\gamma^\ast$, $\gamma^\ep$ reaches one of $[x_i-2\delta,x_i+2\delta]$, $i=1,\ldots,I$ at most within time $T>0$ estimated as (\ref{estimate}). Hence,  we have with (\ref{41}),
\begin{eqnarray*}
v^\ep(x)-v^\ast(x)&\le& \int^0_{-T}(e^{\ep s}-1) \{L(\gamma^\ast(s),\gamma^\ast{}'(s))-c\gamma^\ast{}'(s) \}ds \\
&&+ (e^{-\ep T}-1)v^\ep(\gamma^\ast(-T))+ v^\ep(\gamma^\ast(-T))-v^\ast(\gamma^\ast(-T))\\
&\le&\alpha_3 (\ep T^2+ \ep T +\delta),\\ 
v^\ep(x)-v^\ast(x)&\ge& \int^0_{-T}(e^{\ep s}-1) \{L(\gamma^\ep(s),\gamma^\ep{}'(s))-c\gamma^\ep{}'(s) \}ds \\
&&+ (e^{-\ep T}-1)v^\ep(\gamma^\ep(-T))+ v^\ep(\gamma^\ep(-T))-v^\ast(\gamma^\ep(-T))\\
&\ge&-\alpha_3 (\ep T^2+ \ep T +\delta).
\end{eqnarray*} 
Therefore, we obtain 
\begin{eqnarray*}
|v^\ep(x)-v^\ast(x)|\le \alpha_3 (\ep T^2+ \ep T +\delta)  \mbox{\,\,\ on $\T\setminus K^{c,\ep}$.}
\end{eqnarray*}
Let $x$ be an arbitrary point of $K^{c,\ep}$. Then, there is $\tilde{x}\neq K^{c,\ep} $ such that $|x-\tilde{x}|\le \delta$ and 
\begin{eqnarray*}
|v^\ep(x)-v^\ast(x)|&=& |v^\ep(x)-v^\ast(x)-(v^\ep(\tilde{x})-v^\ast(\tilde{x}))+(v^\ep(\tilde{x})-v^\ast(\tilde{x}))|\\
&\le &|v^\ep(x)-v^\ep(\tilde{x})|+|v^\ast(x)-v^\ast(\tilde{x})|+|v^\ep(\tilde{x})-v^\ast(\tilde{x})|\\
&\le&\alpha_4 \delta + \alpha_3 (\ep T^2+ \ep T +\delta).
\end{eqnarray*}
Thus, we conclude that 
\begin{eqnarray*}
|v^\ep(x)-v^\ast(x)|\le \alpha_5 (\ep T^2+\delta)\le \alpha_6(\ep |\log\delta|^2+\delta)  \mbox{ \,\,\,on $\T$,}
\end{eqnarray*}
where we may take any $\delta=\delta(\ep)$ which tends to  $0$ with $\delta(\ep)\ge\alpha_1\ep$.

Now we deal with the general case. Let $J_i\subset (0,\ep_0)$, $i=1,\ldots,I$ be such that if $\ep\in J_i$ we have $x_i^{\ep,c}\in \mathcal{M}_\alpha^\ep(c )$. If  $\inf J_i=0$, we have $v^\ast(x_i)=0$ due to (iii) of Theorem \ref{d-Mather1}.  Let $\ep_\ast>0$ be the minimum value of $\{\inf J_i\,|\,i\mbox{ is such that $\inf J_i>0$}\}$. If there is no such $i$, re-define $\ep_\ast$ as $\ep_\ast:=\ep_0$. Note that, if $x^{c,\ep}_i\in \mathcal{M}_\alpha^{\ep}(c )$ for some  $\ep\in(0,\ep_\ast)$, we have $v^\ast(x_i)=0$.  For each sufficiently small $\ep\in(0,\ep_\ast)$, there exists at least one $i$ for which $x_i^{c,\ep}\in \mathcal{M}_\alpha^\ep(c )$ and $v^\ast(x_i)=0$. Suppose that $x_{i+1}^{c,\ep}\not\in\mathcal{M}_\alpha^\ep(c )$. Then, one of the following two cases hold to be true: 
\begin{itemize}
\item[(i)] $c+v_x^\ep(x_{i+1}^{c,\ep})$ is negative and hence  $c+v_x^\ep(x)$ is negative for all $x\in[x_{i+1}^{c,\ep}, x_{i+2}^{c,\ep})$,  
\item [(ii)] $c+v_x^\ep(x_{i+1}^{c,\ep})$ is positive and hence $c+v_x^\ep(x)$ is positive for all $x\in(x_i^{c,\ep}, x_{i+1}^{c,\ep})$,
\end{itemize}
where we carefully  note that discontinuity of $c+v^\ep_x$ is allowed to be jump down only, and so is discontinuity of  $c+v^\ast_x$.  Since  $c+v_x^\ep$ converges to $c+v^\ast_x$ pointwise a.e. and hence uniformly except an arbitrarily small neighborhood of points of discontinuity as $\ep\to0+$, the sign of $c+v^\ep_x$ is identical with that of  $c+v^\ast_x$  except a small neighborhood of points of discontinuity and zero points of $c+v^\ep_x$ and $c+v^\ast_x$. 

%Hence we see that (i) (resp. (ii)) holds if $c+v_x$ is negative (resp.   is positive or changes the sign) on $x\in(x_{i+1}^{c,\ep}, x_{i+2}^{c,\ep})$.   

If (i) is the case: We have $x_{i+k}^{c,\ep}\in\mathcal{M}_\alpha^\ep(c)$ such that $\graph(c+v_x^\ep)$ within $[x_{i+1}^{c,\ep},x_{i+k}^{c,\ep}]$ coincides with the unstable manifold growing from  $(x_{i+k}^{c,\ep},0)$, which implies that $v^\ast(x_{i+k})=0$ and $\graph(c+v^\ast_x)$  within $[x_{i+1},x_{i+k}]$ coincides with the stable/unstable manifolds lying on $p\le0$ without any discontinuity. Hence, we obtain by the same procedure as the above, 
$$|v^\ep(x)-v^\ast(x)|\le \alpha_7 (\ep T^2+ \ep T +\delta)\mbox{ \,\,\, on $[x_{i+k-1}+2\delta,x_{i+k}+2\delta)$}.$$
Then we obtain by Lipschitz continuity of $v^\ep$ and $v^\ast$, 
$$|v^\ep(x)-v^\ast(x)|\le \alpha_7 (\ep T^2+ \ep T +\delta)+\alpha_8\delta \mbox{ \,\,\, on $[x_{i+k-1}-2\delta,x_{i+k-1}+2\delta]$},$$
which yields the estimate within  $[x_{i+k-2}+2\delta,x_{i+k-1}-2\delta)$.
We may repeat this argument with the above treatment in $[x_i,x_{i+1}-2\delta]\cap K^{c,\ep}$ if $K^{c,\ep}$ is non-empty, and obtain 
$$|v^\ep(x)-v^\ast(x)|\le \alpha_9 (\ep T^2+ \ep T +\delta) \mbox{ \,\,\, on $[x_{i},x_{i+k}]$}.$$
\indent If (ii) is the case: We have $x_{i+k}^{c,\ep}\in\mathcal{M}_\alpha^\ep(c)$ for which $\graph(c+v_x^\ep)$ within $[x_i^{c,\ep},\bar{x}^{c,\ep}]$ 
with some $\bar{x}^{c,\ep}\in(x_{i+k-1}^{c,\ep},x_{i+k}^{c,\ep}]$ coincides with the unstable manifold growing from $(x_i^{c,\ep},0)$ and, if $\bar{x}^{c,\ep}\neq x_{i+k}^{c,\ep}$, the graph has  discontinuity at $\bar{x}^{c,\ep}$ switching to the unstable manifold of $(x_{i+k}^{c,\ep},0)$ lying on $p\le0$. 
This implies that $\graph(c+v^\ast_x)$  within $[x_{i},\bar{x}]$ with some $\bar{x}\in(x_{i+k-1},x_{i+k}]$  coincides with  the stable/unstable manifolds lying on $p\ge0$ without any discontinuity and, if $\bar{x}\neq x_{i+k}$, the graph  has discontinuity at $\bar{x}$ switching to the unstable manifold of $(x_{i+k},0)$ lying on $p\le0$. Then we obtain with the same procedure as the above,    
 $$|v^\ep(x)-v^\ast(x)|\le \alpha_{10} (\ep T^2+ \ep T +\delta)\mbox{ \,\,\, on $[x_{i},x_{i+1}+2\delta)$},$$
 and by repeating this argument with the same treatment in $[x_{i+k-1},x_{i+k}-2\delta]\cap K^{c,\ep}$ if $K^{c,\ep}$ is non-empty,  
 $$|v^\ep(x)-v^\ast(x)|\le \alpha_{11} (\ep T^2+ \ep T +\delta)\mbox{ \,\,\, on $[x_{i},x_{i+k}]$}.$$
 
 In this way, we conclude $|v^\ep(x)-v^\ast(x)|\le \alpha_{12}(\ep |\log\delta|^2+\delta)$  on $\T$, where we may take any $\delta=\delta(\ep)$ which tends to  $0$ with $\delta(\ep)\ge\alpha_1\ep$. 
\end{proof}
%%%%%%%%%%%%%%%%%%%
%%%%%%%%%%%%%%%%%%%
\subsection{Error estimate in (C2)}
Suppose that a viscosity solution $v$ of (HJ) admits a KAM $n$-torus. This means the following:  $\graph (c+v_x)$ is a smooth $\phi^s_H$-invariant torus on which the Hamiltonian dynamics (\ref{HS}) is $C^1$-conjugate to the linear flow on $\T^n$ with a $\nu,\eta$-Diophantine rotation vector $\omega\in\R^n$, i.e.,  there exists $C^1$-embedding $\Phi=(\varphi,\psi):\T^n\to\T^n\times\R^n $ such that 
\begin{eqnarray*}  
&&\Phi(\T^n)=\graph (c+v_x),\\
&&\phi^s_H(x_0,p_0)=\Phi(\omega s+\theta_0) \mbox{ for all $(x_0,p_0)\in\graph (c+v_x)$},\\
&&|\omega\cdot z|\ge\nu(|z_1|+\cdots+|z_n|)^{-\eta} \mbox{ for all $z\in\Z^n\setminus\{0\}$ ($\nu>0,\eta>n-1$)},\\
&&\theta_0 \mbox{ is a constant depending on $(x_0,p_0)$}.
\end{eqnarray*}
Note that in such a case, viscosity solutions $v$ are unique up to constants. Since the linear flow with a Diophantine rotation vector is ergodic on $\T^n$, each minimizing curve $\gamma^\ast(s)$ for $v(x_0)$, satisfying $\gamma^\ast(s)\equiv\varphi(\omega s+\theta_0)$, is also ergodic on $\T^n$. According to \cite{Dumas},\cite{BGW}, for each $\delta>0$ the set 
$$\mathcal{N}_\delta:=\{  \omega s +\theta_0 \,|\,-\frac{\beta_0}{\delta^\eta}\le s\le0\}\mbox{ ($\beta_0>0$ is a constant)}$$
is $\delta$-dense in $\T^n$, i.e., 
$$\bigcup_{\zeta\in\mathcal{N}_\delta} B_\delta(\zeta)=\T^n,$$
where $B_\delta(\zeta)$ is the $\delta$-ball of around $\zeta$. Hence, along only one minimizing curve $\gamma^\ast(s)$ within $-\frac{\beta_0}{\delta^\eta}\le s\le0$, we obtain the whole information on $\T^n$ in ``$\delta$-accuracy". This strategy is first used by Bessi \cite{Bessi} in the vanishing viscosity method, and by Soga \cite{Soga3} in the finite difference approximation. 
%%%%%%%%%%%%
\begin{Thm}\label{4321}
Let $v$ be a viscosity solution of (\ref{HJ}) corresponding to a KAM torus with a $\nu,\eta$-Diophantine rotation vector. Let $v^\ep$ be  the unique viscosity solution of (\ref{dHJ}). Then, adding a constant which may depend on $\ep$ to $v$ if necessary, we have for $\ep\to0$,
$$\norm v^\ep-v\norm_{C^0(\T^n)}\le \beta\varepsilon^{\frac{1}{1+2\eta}},$$ 
where $\beta>0$ is a constant independent of $\ep$.
\end{Thm}
%%%%%%%%%%%%
\begin{Rem}  There exists $v^\ast=\lim_{\ep\to0+}v^\ep$, which is  identical with $v$ up to constant. The theorem  does not necessarily mean that the error between $v^\ep$ and $v^\ast$  is estimated by $\beta\varepsilon^{\frac{1}{1+2\eta}}$. The possible added constant also goes to $0$ as $\ep\to0$. 
\end{Rem}
%%%%%%%%%%%%%%%%
\begin{proof}[{\it Proof of Theorem {\rm\ref{4321}}}]
In what follows, $\beta_i$ are some positive constants. Adding a constant if necessary, we have 
\begin{eqnarray}\label{411}
v^\ep(x)-v(x)\le0\mbox{ in $\T^n$, \,\,\,}
v^\ep(x^\ast)-v(x^\ast)=0\mbox{ for some $x^\ast=x^\ast(\ep)\in\T^n$}.
\end{eqnarray}
Let $\gamma^\ast$ be a minimizing curve for $v(x^\ast)$. We see that 
\begin{eqnarray*}
v(x^\ast)&=& \int^0_{-\tau}  ( L(\gamma^\ast(s),\gamma^\ast{}'(s))-c\cdot\gamma^\ast{}'(s) +h(c ) ) ds+v(\gamma^\ast(-\tau)),\\
v^\ep(x^\ast)&\le& \int^0_{-\tau}  e^{\ep s}( L(\gamma^\ast(s),\gamma^\ast{}'(s))-c\cdot\gamma^\ast{}'(s) +h(c ) ) ds+e^{-\ep \tau}v^\ep(\gamma^\ast(-\tau)). 
\end{eqnarray*}
Subtracting $v(x^\ast)$ from $v^\ep(x^\ast)$ yields
\begin{align*}
&v^\ep(x^\ast)-v(x^\ast)\\
\le&\,
\int^0_{-\tau}(e^{\ep s}-1)( L(\gamma^\ast(s),\gamma^\ast{}'(s))-c\cdot\gamma^\ast{}'(s) +h(c ) ) ds
+e^{-\ep \tau}v^\ep(\gamma^\ast(-\tau))-v(\gamma^\ast(-\tau))\\
=&\,v^\ep(\gamma^\ast(-\tau))-v(\gamma^\ast(-\tau))\\
&+\underline{\int^0_{-\tau}(e^{\ep s}-1)( L(\gamma^\ast(s),\gamma^\ast{}'(s))-c\cdot\gamma^\ast{}'(s) +h(c ) ) ds
+(e^{-\ep \tau}-1)v^\ep(\gamma^\ast(-\tau))}_{(\sharp)}.
%v^\ep(x^\ast)-v(x^\ast)&\le&\int^0_{-\tau}(e^{\ep s}-1)( L(\gamma^\ast(s),\gamma^\ast{}'(s))-c\cdot\gamma^\ast{}'(s) +h(c ) ) ds\\
%&&+e^{-\ep \tau}v^\ep(\gamma^\ast(-\tau))-v(\gamma^\ast(-\tau))\\
%&=&v^\ep(\gamma^\ast(-\tau))-v(\gamma^\ast(-\tau))\\
%&&+\underline{\int^0_{-\tau}(e^{\ep s}-1)( L(\gamma^\ast(s),\gamma^\ast{}'(s))-c\cdot\gamma^\ast{}'(s) +h(c ) ) ds}\\
%&&\underline{+(e^{-\ep \tau}-1)v^\ep(\gamma^\ast(-\tau))}{}_{(\sharp)}. 
\end{align*}
Since $|  L(\gamma^\ast(s),\gamma^\ast{}'(s))-c\cdot\gamma^\ast{}'(s) +h(c )|\le \beta_1$ for all $s\in\R$, $|v^\ep|\le\beta_2 $ independently of $\ep$ and $e^{\ep s}\le1+\ep s +\frac{(\ep s)^2}{2}$, $e^{\ep s}\ge1+\ep s $ for all $s\le0$, we obtain  
\begin{eqnarray*}
|(\sharp)|\le\beta_1\int^0_{-\tau}(1-e^{\ep s})ds+\beta_2(1-e^{-\ep \tau})\le \beta_1\frac{\ep \tau^2}{2}+\beta_2\ep\tau\mbox{\,\,\, for any $\tau>0$}.
 \end{eqnarray*}
 By (\ref{411}), we obtain $-(\sharp)\le v^\ep(\gamma^\ast(-\tau))-v(\gamma^\ast(-\tau))\le 0$, i.e., 
 \begin{eqnarray*}\label{422}
 |v^\ep(\gamma^\ast(-\tau))-v(\gamma^\ast(-\tau))|\le \max\{\beta_3\ep\tau^2, \beta_3\ep\tau\}\mbox{\,\,\,\, for any $\tau\ge0$}.
 \end{eqnarray*}
 Note that $\Phi=(\varphi,\psi)$ is an embedding and hence $\varphi:\T^n\to\T^n$ is diffeomorphic. Let $x\in\T^n$ be an arbitrary point. Set $\zeta_x:=\varphi^{-1}(x)$. Since $\mathcal{N}_\delta$ is $\delta$-dense in $\T^n$, we have $\zeta\in\mathcal{N}_\delta $ such that $\zeta_x\in B_\delta(\zeta)$. For this $\zeta$, there exists $s\in[-\frac{\beta_0}{\delta^\eta},0]$ such that 
 $$\zeta=\omega s+\theta_0=\varphi^{-1}(\gamma^\ast(s)).$$ 
 Therefore, we see that 
\begin{align*}
&|v^\ep(x)-v(x)|\\
\le&\, 
|v^\ep(x)-v^\ep(\gamma^\ast(s))|+|v^\ep(\gamma^\ast(s))-v(\gamma^\ast(s))|+|v(\gamma^\ast(s))-v(x)|\\
=&\, |v^\ep(\gamma^\ast(s))-v(\gamma^\ast(s))|
+|v^\ep(\varphi(\zeta_x))-v^\ep(\varphi(\zeta))|+|v(\varphi(\zeta_x))-v(\varphi(\zeta))|\\
\le&|v^\ep(\gamma^\ast(s))-v(\gamma^\ast(s))|+\beta_4|\zeta_x-\zeta|\\
\le&\max\{ \beta_5(\frac{\ep}{\delta^{2\eta}}+\delta),  \beta_5(\frac{\ep}{\delta^{\eta}}+\delta) \}
=\beta_5(\frac{\ep}{\delta^{2\eta}}+\delta).
\end{align*}
%\begin{eqnarray*}
%|v^\ep(x)-v(x)|&\le& |v^\ep(x)-v^\ep(\gamma^\ast(s))|+|v^\ep(\gamma^\ast(s))-v(\gamma^\ast(s))|+|v(\gamma^\ast(s))-v(x)|\\
%&=& |v^\ep(\gamma^\ast(s))-v(\gamma^\ast(s))|\\
%&&+|v^\ep(\varphi(\zeta_x))-v^\ep(\varphi(\zeta))|+|v(\varphi(\zeta_x))-v(\varphi(\zeta))|\\
%&\le&|v^\ep(\gamma^\ast(s))-v(\gamma^\ast(s))|+\beta_4|\zeta_x-\zeta|\\
%&\le&\max\{ \beta_5(\frac{\ep}{\delta^{2\eta}}+\delta),  \beta_5(\frac{\ep}{\delta^{\eta}}+\delta) \}\\
%&=& \beta_5(\frac{\ep}{\delta^{2\eta}}+\delta).
%\end{eqnarray*}
Taking $\delta=\ep^\frac{1}{1+2\eta}$, we obtain the conclusion. 
\end{proof}

\end{document}